\documentclass[12pt]{article}
\usepackage{amsmath}
\usepackage{amssymb}
\usepackage{mathrsfs}
\newcommand{\N}{\mathbb{N}}
\newcommand{\R}{\mathbb{R}}
\newcommand{\C}{\mathbb{C}}

\newcommand{\beq}{\begin{eqnarray}
  }
\newcommand{\eeq}{\end{eqnarray}}
\newcommand{\beqst}{\begin{eqnarray*}
  }
\newcommand{\eeqst}{\end{eqnarray*}}
\newtheorem{theorem}{Theorem}

\newtheorem{lemma}{Lemma}
\newtheorem{corollary}{Corollary}

\title{Uniqueness of simultaneous reconstruction of general space- and time-dependent sources and initial states in fractional diffusion 
equations and systems from single boundary measurements. }
 \author{Jaan Janno
 }
 \date{}

\begin{document}
 
 \maketitle

{\begin{center} {\small Tallinn University of Technology, Department of Cybernetics\\
 Ehitajate tee 5, 19086 Tallinn, Estonia\\
 E-mail: jaan.janno@taltech.ee
 }
 \end{center}

 \begin{abstract}
Inverse problem to determine simultaneously  a general space- and time-dependent source and an initial state in a fractional diffusion equation from an {\it a posteriori} measurement 
of the normal derivative of the state on a portion of a boundary of the space domain is considered.  Uniqueness for this problem is 
proved under the assumption that an order of a fractional derivative involved in the equation is irrational. The uniqueness result is generalized 
to an inverse problem for a coupled system of fractional diffusion equations, where both sources and initial states are unknown and first component of the state is measured on 
the boundary.
 \end{abstract}



\section{Introduction}

Fractional diffusion equations (FDE) are widely used to model anomalous diffusion processes in physics, biology, engineering sciences etc. Fractional time 
derivatives contained in these equations bring along additional degrees of freedom that can be adjusted in order better to match  mathematical models to real processes. 
Parameters of the equations can be determined via solution of inverse problems that involve overdetermination (observation) conditions. 

The fractional time derivative involves history of an input function. 
Therefore, inverse problems for FDE have some features that does not occur in the case of classical diffusion equations (CDE) that involve integer time derivatives. 

For instance, the solution of FDE depends on a history of a state before a beginning of the observed process. This may be unknown in practice.
In inverse problems to determine parameters of the equations the unknown history can be removed provided there is a gap between 
a shutdown of an unknown source and a generated source. This can be proved by means of analyticity arguments \cite{Janno2024}. 

Orders of fractional derivatives can be independently recovered even if some
other parameters of models are unknown. 
In the case of time trace observations, independent formulas for the orders follow in the 
process $t\to 0$ \cite{Hul,J2016,Kra} or equivalently in the process $s\to\infty$ \cite{JinKian}, where  $s$ is the argument of the Laplace transform. 

A number of papers deal with inverse problems to reconstruct source functions of the form $g(t)f(x)$, where $t$ is the time and $x$ is a space variable. Problems to determine $f$ from 
the data $u(T,x)$, where $u$ is a state variable and  $T$ is a final time, or  $g$ from the data $\Phi u(t,\cdot)$, where $\Phi$ is a functional acting on functions that depend on $x$, 
are rather classical and methods used in the case of CDE can be adjusted to the case of FDE \cite{Ism,Jin,KJ,Kir,Lop,Saka,Slo}. 
Besides, the time-dependent source factor $g$ in FDE can be uniquely reconstructed from the final data $u(T,x)$.  A proof of this assertion
uses asymptotic power-type expansions of Mittag-Leffler functions involved in a formula of the state \cite{JFCAA}. Uniqueness for the similar problem in the case of CDE is open.  
\iffalse The mentioned asymptotics is used to prove uniqueness 
for direct problems for FDE where boundary conditions are replaced by final condition \cite{j2021}. Uniqueness for such problems for CDE is also open.
In case of final observation, independent equations 
for the orders can be extracted from high frequency modes \cite{JK,Kian1,Kaif}. \fi

A function in an interval $(0,t_0)$ can be uniquely recovered from its values and values of its fractional derivative in the interval $(t_0,t_1)$, where 
$0<t_0<t_1<\infty$ (obviously, this is not valid in the case of integer derivative). 
A proof of  uses the fact that the Laplace transform of the kernel of fractional derivative is not meromorphic \cite{KJ1}. 
This result and similar techniques can be applied to study inverse problems to recover parameters of equations from data given in $(t_0,t_1)\times\Gamma$, where  
$(t_0,t_1)$ is an {\it a posteori} time interval and $\Gamma$ is a portion 
of the boundary of a space domain $\Omega$.

For instance, in \cite{Kian2} simultaneous reconstruction of several parameters in FDE from data given in the set $(t_0,t_1)\times\Gamma$
is shown. The data consist of  ${\partial\over\partial \nu}u$ and ${\partial\over\partial \nu}Au$, where $\nu$ is the normal vector 
and $A$ is an elliptic operator. 

Problems to reconstruct the source term of the form $g(t)f(x)$ in FDE from observations on  $\Gamma$ 
 are studied in \cite{JannoKian}. It is shown that in case $g(t)=0$, $t>t_0$, 
the function $g$ is uniquely recovered by ${\partial\over\partial \nu}u(t,x_0)$, $t\in (t_0,t_1)$, where $x_0$ is a fixed point of $\Gamma$, and $f(x)$, $x\in\Omega$,
 is uniquely determined by  the data
\beq\label{data}
{\partial\over\partial \nu}u(t,x),\;t\in (t_0,t_1),\;x\in\Gamma.
\eeq
 The proof operates with a jump at a branch line of the Laplace transform of an observation function. 
Similar results with a different method are obtained in \cite{yama}. These results are extended to interior measurements, too \cite{Yu}.

In case the order of the fractional derivative in FDE is irrational, then the inverse problem is even more informative. Then the full source function 
$g(t)f(x)$ is uniquely recovered from  the data \eqref{data}
 \cite{JannoKian,Lor}. 

In the present paper we will show that in the case of irrational order of the fractional derivative, even a general time and space-dependent source function $f(t,x)$,
$(t,x)\in (0,t_0)\times\Omega$,  
in FDE such that $f=0$, $t>t_0$, and an initial state can be 
uniquely reconstructed from the data \eqref{data}.
We will also extend this result to a coupled system of FDEs for a pair of states $(u,v)$ 
that contain two unknown sources $f(t,x)$ and $\chi(t,x)$. Both sources vanishing for $t>t_0$ and both initial states are 
uniquely reconstructed by the measurement \eqref{data} of the first component $u$. The measure of $\Gamma$ and the length of the interval $(t_0,t_1)$ 
may be arbitrarily small.  To prove these results, we will use an analytic extension of a jump of of the Laplace transform 
of the observation function to infinite number of branches. Taking a certain limit over these branches, 
we can introduce an additional continuous variable to a basic equation. Operating with this variable, we will reach the desired result. 

To the author's knowledge, this is the first paper where uniqueness for an inverse problem to determine a general function defined in a domain of higher dimension from a 
function given in a domain of lower dimension is proved.

\section{Formulation of problems}

Let $\Omega\subset\R^N$ be a bounded domain with the boundary $\partial\Omega$. In the case $N\ge 2$ we assume that $\partial\Omega$ is sufficiently regular. In the case 
$N=1$ we define $\Omega=(0,\pi)$. 

Let us formulate the following initial-boundary value problem (direct problem) 
for a coupled system of fractional diffusion equations involving a pair of unknowns $(u,v)$:

\beq\label{e1}
&&\hskip-1.3truecm D_t^\alpha (u-\varphi)(t,x)-\kappa\Delta u(t,x)+cu(t,x)+av(t,x)=f(t,x),\; x\in\Omega,\; t>0,
\\[1ex] \label{e2}
&&\hskip-1.3truecm D_t^\alpha (v-\psi)(t,x)-\varkappa\Delta v(t,x)+dv(t,x)+bu(t,x)=\chi(t,x),\; x\in\Omega,\; t>0,
\eeq
\beq\label{e3}
&&u(t,x)=0,\; x\in\partial\Omega,\; t>0,
\\ \label{e4}
&&v(t,x)=0,\; x\in\partial\Omega,\; t>0,
\eeq
\beq \label{e5}
&&u(0,x)=\varphi(x),\; x\in\Omega,
\\ \label{e6}
&&v(0,x)=\psi(x),\;x\in\Omega,
\eeq
where $0<\alpha<1$, $D_t^\alpha$ is the Riemann-Liouville fractional derivative of the order $\alpha$, i.e.
$$
D_t^\alpha q(t)={d\over dt}\int_0^t {(t-\tau)^{-\alpha}\over \Gamma(1-\alpha)}q(\tau)d\tau,
$$
and $\Delta$ is the Laplacian.
The 
parameters $\kappa,\varkappa,a,b,c,d$ are given scalars satisfying the inequalities
\beq\label{kappaineq}
\kappa, \varkappa>0,\;\; c,d\ge 0.
\eeq

We mention that the terms $D_t^\alpha (u-\varphi)(t,x)$ and $D_t^\alpha (v-\psi)(t,x)$  in \eqref{e1} and \eqref{e2} 
are the Caputo fractional derivatives of 
$u$ and $v$, respectively.

\smallskip
If $a=0$ then \eqref{e1}, \eqref{e3} and \eqref{e5} form an independent problem for $u$. 

\medskip
In the inverse problems we will use an additional condition that is given in a portion of boundary $\Gamma\subset\partial\Omega$. 
In the case $N\ge 2$ we assume that $\Gamma$ is open and in the case $N=1$ we define  $\Gamma=\{0\}$. 
The additional condition has the following 
form:
\beq\label{e7}
{\partial\over \partial\nu}u(t,x')=h(t,x'),\; t\in (t_0,t_1),\; x'\in\Gamma,
\eeq
where $\nu$ is the outer normal of $\partial\Omega$, the numbers $t_0,t_1$ are such that $0<t_0<t_1<\infty$ and $h$ is a given function. 
We will be limited to the case when the source(s) vanish after $t=t_0$.
We are going to study two  inverse problems (IP1 and IP2) formulated below.

\medskip
\noindent IP1: Let $a=0$. Find $f(t,x)$,  $\varphi(x)$ and $u(t,x)$ such that  $f(t,\cdot)=0$, $t>t_0$, and \eqref{e1}, \eqref{e3}, \eqref{e5},
 \eqref{e7} are valid.

\medskip\noindent
 IP2: Let $a\ne 0$. Find $f(t,x),\chi(t,x)$, $\varphi(x),\psi(x)$ and $u(t,x)$, $v(t,x)$ such that $f(t,\cdot)=\chi(t,\cdot)=0$, $t>t_0$,
 and \eqref{e1} - \eqref{e6}, \eqref{e7} are valid.

\smallskip Plan of the paper is as follows.
In the next section we will give some basic definitions and relations and in Section \ref{secdirect} we will prove the existence and uniqueness for the direct problem.
Section \ref{secinv} contains uniqueness results for the inverse problems.

\section{Functional spaces}

Let us consider the eigenvalue problem   
$-\Delta {\rm z}(x)=\lambda {\rm z}(x)$, $x\in\Omega$, ${\rm z}(x)=0$, $x\in\partial\Omega$.
We denote the family of its solutions by $(\lambda,{\rm z})=(\lambda_k,{\rm z}_{kl})$, $l=1,\ldots,\mu_k$, $k\in\N$. Here $\mu_k$ 
is the multiplicity of $\lambda_k$. Assume that
the eigenvalues are monotonically ordered, i.e. $0<\lambda_1<\lambda_2<\ldots$, 
and
the system of eigenfunctions ${\rm z}_{kl}$,  $k\in\N$, $l=1,\ldots,\mu_k$, is orthonormed in $L_2(\Omega)$. Then 
 any $w\in L_2(\Omega)$ can be expanded into Fourier series $w=\sum\limits_{k=1}^\infty \sum\limits_{l=1}^{\mu_k}\langle w,{\rm z}_{kl}\rangle_{L_2(\Omega)}{\rm z}_{kl}$
that converges in $L_2(\Omega)$. 

In the sequel we will denote the  Fourier coefficients by subscripts, i.e $w_{kl}=\langle w,{\rm z}_{kl}\rangle_{L_2(\Omega)}$, 
$w\in L_2(\Omega)$. 
This means that the notation of Fourier coefficients of functions involved in the direct problem is as follows:
\beqst
&&u_{kl}(t)=\langle u(t,\cdot),{\rm z}_{kl}\rangle_{L_2(\Omega)},\; v_{kl}(t)=\langle v(t,\cdot),{\rm z}_{kl}\rangle_{L_2(\Omega)},\;
\\
&&f_{kl}(t)=\langle f(t,\cdot),{\rm z}_{kl}\rangle_{L_2(\Omega)},\; \chi_{kl}(t)=\langle \chi(t,\cdot),{\rm z}_{kl}\rangle_{L_2(\Omega)},\;
\\
&&\varphi_{kl}=\langle \varphi,{\rm z}_{kl}\rangle_{L_2(\Omega)},\; \psi_{kl}=\langle \psi,{\rm z}_{kl}\rangle_{L_2(\Omega)}.
\eeqst

Let the domain of $\Delta$ be
$$
H_0^2(\Omega)=\{w\in H^2(\Omega)\, :\, w(x)=0,\, x\in\partial\Omega\}.
$$
In addition, let us introduce the following family of sets:
$$
{\mathcal D}^\rho=\{w\in L_2(\Omega)\, :\, \sum_{k=1}^\infty \sum_{l=1}^{\mu_k}\lambda_k^{2\rho} |w_{kl}|^2<\infty\},\; \rho\ge 0.
$$
These are Hilbert spaces with the scalar products $\langle w,\xi\rangle_{{\mathcal D}^\rho}=\sum\limits_{k=1}^\infty \sum\limits_{l=1}^{\mu_k}
\lambda_k^{2\rho}w_{kl} {\overline \xi}_{kl}$. 

It hold ${\mathcal D}^1=H_0^2(\Omega)$ and  ${\mathcal D}^{\rho_2}\hookrightarrow {\mathcal D}^{\rho_1}$ for $0\le\rho_1\le\rho_2$.

\medskip
Next, let $X$ be a complex Banach space with a norm $\|\cdot\|$ and $S\subseteq\R$. We introduce the following spaces of  functions 
that map $S$ to $X$:
\beqst
&&L_p(S;X)=\left\{w \, :\, \|w\|_{L_p(S;X)}:=\left[\int_S\|w(t)\|^pdt\right]^{1/p}<\infty\right\},\;\; 1\le p<\infty,\\
&&L_\infty(S;X)=\left\{w \, :\, \|w\|_{L_\infty(S;X)}:={\rm ess}\, \sup_{t\in S}\|w(t)\|<\infty\right\},
\\
&&W_p^l(S;X)=\left\{w \, :\, {d^i\over dt^i}w\in {L_p(S;X)},\, i=0,\ldots,l\right\},\; 1\le p\le \infty,\; l\in\N,
\\
&&C(S;X)=\{w\, :\, w - \mbox{continuous in $S$}\}.
\eeqst
If $S$ is compact, then $C(S;X)$ is a Banach space with the norm $\|w\|_{C(S;X)}=\max\limits_{t\in S}\|w(t)\|$. 

Further, let $T\in (0,\infty)$, $1<p<\infty$,  and introduce the following space:
\beqst
&&\hskip-7truemm
_0H_p^\alpha ((0,T);X)=\{w|_{(0,T)}\, :\, w\in H_p^\alpha (\R;X),\, {\rm supp}\,w \subseteq [0,\infty)\},
\eeqst
where $
H_p^\alpha (\R;X)=\{w\in L_p(\R;X)\, :\, {\mathcal F}^{-1}|\xi|^{\alpha}{\mathcal F}w\in L_p(\R;X)\}$
and ${\mathcal F}$ is the Fourier transform with the argument $\xi$.

In case $X=\C$ we drop the value space $\C$ in these notations.

\smallskip
The Riemann-Liouville fractional integral of the order $\beta>0$ is given by the following formula:
$$
I_t^\beta q(t)=\int_0^t {(t-\tau)^{\beta-1}\over\Gamma(\beta)}q(\tau)d\tau,\; t>0.
$$
Let us formulate a lemma that follows from Corollary 2.8.1 in \cite{zachdiss} and provides a relation between $D_t^\alpha$ and $I_t^\alpha$.

\begin{lemma}\label{lemma31} 
Let $X$ be a complex Hilbert space and
  $p\in (1,\infty)$.
The operator  $I_t^{\alpha}$ is a
bijection from $L_p((0,T);X)$ onto $_{0}H_p^\alpha((0,T);X)$, the inverse of $I_t^\alpha$ is the Riemann-Liouville
fractional derivative $D_t^\alpha$
and
$$\|w\|_{_0H_p^\alpha((0,T);X)}=\|{D_t^\alpha } w\|_{L_p((0,T);X)}$$
is a norm in the space $_{0}H_p^\alpha((0,T);X)$.
Moreover,  in case
$p\in ({1\over \alpha},\infty)$ it holds\break\\[-2ex] $_0H_p^\alpha((0,T);X)\hookrightarrow C([0,T];X)$
 and $w(0)=0$ for $w\in {_{0}H_p^\alpha}((0,T);X).$
\end{lemma}

\section{Direct problem}\label{secdirect}

Let us start by introducing some additional notation.

\smallskip
Throughout the paper,  $z^\beta$, $z\in\C$, $\beta>0$, denotes the principal value of $\beta$-th power of  $z$, i.e. 
$z^\beta=|z|^\beta e^{i\beta {\rm Arg} z}$, ${\rm Arg}z\in (-\pi,\pi]$. 

\smallskip
Let $X$ be a Banach space and $q\in L_1((0,T);X)$ for any $T>0$. Let there exist $\sigma\in\R$ such that $e^{-\sigma t}q\in L_1((0,\infty);X)$. Then it is well-known that the Laplace transform 
 $$
({\mathcal L}q)(s)=\int_0^\infty e^{-st}q(t)dt
 $$
exists and is analytic as a function with values in $X$ in the half-plane ${\rm Re}s>\sigma$. 
We will also use capitals to denote the Laplace transforms. For instance,
\beqst
&&U_{kl}(s)=({\mathcal L}u_{kl})(s),\; V_{kl}(s)=({\mathcal L}v_{kl})(s),\;\\
&&F_{kl}(s)=({\mathcal L}f_{kl})(s),\; {\mathcal X}_{kl}(s)=({\mathcal L}\chi_{kl})(s).
\eeqst

Next let us introduce the following spaces for the components $u$ and $v$ of the direct problem \eqref{e1} - \eqref{e6}:
\beqst
&&\hskip-1truecm \mathcal{W}_{\alpha,p,T}=\left\{w\in C([0,T];L_2(\Omega))\cap L_p((0,T);H_0^2(\Omega))\, :\,\right.
\\
&&\qquad \left. w-w(0,\cdot)\in {_0H_p^{\alpha}}((0,T);L_2(\Omega))\right\},
\;\; p>1,\, T>0,
\\
&&\hskip-1truecm\mathcal{W}_{\alpha,p}=\{w\, :\, [0,\infty)\to L_2(\Omega)\, :\, w|_{[0,T]\times\Omega}\in \mathcal{W}_{\alpha,p,T}\;\; \forall T>0\},\; p>1.
\eeqst
We mention that if a solution of \eqref{e1} - \eqref{e6} belongs to $(\mathcal{W}_{\alpha,p})^2$ for some $p\in (1,\infty)$, 
then it is a strong solution, i.e. all terms in \eqref{e1} - \eqref{e6} are regular distributions.

\begin{lemma}\label{dplem1}
Let \eqref{e1} - \eqref{e6} have a solution $(u,v)\in (\mathcal{W}_{\alpha,p})^2$ for some $p\in (1,\infty)$. Moreover,
let there exist $\sigma\in\R$ such that $e^{-\sigma t}u, e^{-\sigma t}v\in L_1((0,\infty);L_2(\Omega))$. Then 
$(u_{kl},v_{kl})\in (C[0,\infty))^2$, $e^{-\sigma t}u_{kl},e^{-\sigma t}v_{kl}\in L_1(0,\infty)$, $k\in\N$, $l=1,\ldots,\mu_k$. 
Moreover, 
\beq\label{prop1}
&&\hskip-12truemm U_{kl}(s)={(s^\alpha\!+\!\varkappa\lambda_k\!+\!d)(F_{kl}(s)+s^{\alpha-1}\varphi_k)-
a({\mathcal X}_{kl}(s)+s^{\alpha-1}\psi_{kl})\over (s^\alpha+\kappa\lambda_k+c)(s^\alpha+\varkappa
\lambda_k+d)-ab},\, 
\\[1ex]  \label{prop2} 
&&\hskip-12truemm V_{kl}(s)={(s^\alpha\!+\!\kappa\lambda_k\!+\!c)({\mathcal X}_{kl}(s)+s^{\alpha-1}\psi_{kl})-b(F_{kl}(s)+s^{\alpha-1}\varphi_{kl})\over (s^\alpha+\kappa\lambda_k+c)(s^\alpha+\varkappa\lambda_k+d)-ab},\;
\eeq 
for ${\rm Re}s>\sigma$, $k\in\N$, $l=1,\ldots,\mu_k$. 
\end{lemma}

\noindent{\it Proof}. Since $(u,v)\in (\mathcal{W}_{\alpha,p})^2$, from the relations
 $u_{kl}(t)=\langle u(t,\cdot),{\rm z}_{kl}\rangle_{L_2(\Omega)}$ and $v_{kl}(t)=\langle u(t,\cdot),{\rm z}_{kl}\rangle_{L_2(\Omega)}$ we obtain 
$(u_k,v_k)\in (C[0,\infty))^2$, $T>0$. 
 Next, by Lemma \ref{lemma31} from  \eqref{e1}, \eqref{e2} we deduce 
\beq\label{pro1}
&&\hskip-1.3truecm u(t,x)\!-\!\varphi(x)=I_t^\alpha \left[\kappa\Delta u(t,x)\!-\!cu(t,x)\!-\!av(t,x)\!+\!f(t,x)\right],\, x\in\Omega,\, t>0,
\\ \label{pro2}
&&\hskip-1.3truecm v(t,x)\!-\!\psi(x)=I_t^\alpha \left[\varkappa\Delta v(t,x)\!-\!dv(t,x)\!-\!bu(t,x)\!+\!\chi(t,x)\right],\, x\in\Omega,\, t>0.
\eeq
All terms in these equations restricted to $(0,T)\times \Omega$, $T>0$, belong to $L_p((0,T);L_2(\Omega))$. 
Taking the scalar product of the equations \eqref{pro1} and \eqref{pro2} with ${\rm z}_{kl}$ and moving the integral operator $I_t^\alpha$ under this product we deduce the 
following equations:
\beq\label{pro3}
&&u_{kl}(t)-\varphi_{kl}=I_t^\alpha\left[-(\kappa\lambda_k+c)u_{kl}(t)-av_{kl}(t)+f_{kl}(t)\right], \, t>0,
\\ \label{pro4}
&&v_{kl}(t)-\psi_{kl}=I_t^\alpha\left[-(\varkappa\lambda_k+d)v_{kl}(t)-bu_{kl}(t)+\chi_{kl}(t)\right], \, t>0,
\eeq
for $k\in\N$, $l=1,\ldots,\mu_k$.
Further, the assumption $e^{-\sigma t}u, e^{-\sigma t}v\in\break L_1((0,\infty);L_2(\Omega))$ implies
\beqst
&&\hskip-5truemm e^{-\sigma t}u_{kl}=\langle e^{-\sigma t}u,{\rm z}_{kl}\rangle_{L_2(\Omega)}\in L_1(0,\infty),
e^{-\sigma t}v_{kl}=\langle e^{-\sigma t}v,{\rm z}_{kl}\rangle_{L_2(\Omega)}\in L_1(0,\infty).
\eeqst
Therefore, we can apply Laplace transform to \eqref{pro3}, \eqref{pro4}. We obtain
\beq\label{pro5}
&&U_{kl}(s)-{\varphi_{kl}\over s}={1\over s^\alpha}\left[-(\kappa\lambda_k +c)U_{kl}(s)-aV_{kl}(s)+F_{kl}(s)\right],\; 
\\ \label{pro6}
&&V_{kl}(s)-{\psi_{kl}\over s}={1\over s^\alpha}\left[-(\varkappa\lambda_k +d)V_{kl}(s)-bU_{kl}(s)+{\mathcal X}_{kl}(s)\right],
\eeq
for $\; {\rm Re} s>\sigma$, $k\in\N$, $l=1,\ldots,\mu_k$.
Solving this system for $U_{kl}(s)$ and $V_{kl}(s)$ we reach \eqref{prop1}, \eqref{prop2}. \hfill $\Box$

\medskip
Let us factorize the denominator in \eqref{prop1}, \eqref{prop2} as follows:
\beq\label{factor}
&&(s^\alpha+\kappa\lambda_k+c)(s^\alpha+\varkappa
\lambda_k+d)-ab=(s^\alpha+\breve\lambda_k)(s^\alpha+\hat\lambda_k),
\\[1ex]
\label{lambdad}
&&\left.\begin{split}
&\breve\lambda_k={1\over 2}\left[(\kappa+\varkappa)\lambda_k+c+d+\sqrt{\left((\kappa-\varkappa)\lambda_k+c-d\right)^2+4ab}\right],
\\
&\hat\lambda_k={1\over 2}\left[(\kappa+\varkappa)\lambda_k+c+d-\sqrt{\left((\kappa-\varkappa)\lambda_k+c-d\right)^2+4ab}\right].
\end{split}\right\}
\eeq
\begin{lemma}\label{lambdaklemma} Assume that the following inequalities are valid:
\beq \label{discr}
&&\hskip-5truemm\left((\kappa-\varkappa)\lambda_k+c-d\right)^2+4ab\ge 0,\; k\in\N,
\\[1ex]
\label{paramineq}
&&ab\le\min\{c^2;d^2\}.
\eeq
Then there exist  $c_1,c_2>0$ such that 
\beq\label{lamineq}
c_1\lambda_k\le \hat\lambda_k\le \breve\lambda_k \le c_2\lambda_k,\;\;   k\in\N.
\eeq
\end{lemma}

\noindent{\it Proof}.  Firstly, we prove the left inequality in \eqref{lamineq}. 
From \eqref{lambdad} we have
\beqst
&&\hat\lambda_k\ge {1\over 2}\left[(\kappa+\varkappa)\lambda_k+c+d-\left(\left|(\kappa-\varkappa)\lambda_k+c-d\right|+2\sqrt{\max\{0;ab\}}\right)\right]
\\
&&=\left\{\begin{array}{ll}
\kappa\lambda_k+c-\sqrt{\max\{0;ab\}}\ge \kappa\lambda_k &\mbox{if $(\kappa-\varkappa)\lambda_k+c-d\le 0$},
\\
\varkappa\lambda_k+d-\sqrt{\max\{0;ab\}}\ge \varkappa\lambda_k &\mbox{if $(\kappa-\varkappa)\lambda_k+c-d>0$}.\end{array}\right.
\eeqst
Therefore, $c_1\lambda_k\le \hat\lambda_k$ with $c_1=\min\{\kappa;\varkappa\}$. Evidently, $\hat\lambda_k\le \breve\lambda_k$. Further, using the inequality $1\le {\lambda_k\over\lambda_1}$, $k\in\N$, 
we deduce 
$\breve\lambda_k\le c_2\lambda_k$ with
$c_2= {1\over 2}\left[(\kappa+\varkappa)+{c+d\over\lambda_1}+\sqrt{\left(|\kappa-\varkappa|+{|c-d|\over\lambda_1}\right)^2+{4|ab|\over \lambda_1^2}}\right].
$
\hfill $\Box$

\medskip
Let us define the following sectors on the complex plane:
$$
\Sigma(\eta,\theta)=\{z\in\C\setminus\{\eta\}\, :\,  |{\rm Arg}\, (z-\eta)|< \theta\},\; \eta\in \R,\; \theta\in (0,\pi].
$$

\begin{theorem}\label{dpthm1}
\begin{description}
\item{\rm (i)} Let \eqref{e1} - \eqref{e6} have a solution $(u,v)\in (\mathcal{W}_{\alpha,p})^2$ for some $p\in (1,\infty)$ and 
there exist $\sigma\in\R$ such that $e^{-\sigma t}u,e^{-\sigma t}v \in L_1((0,\infty);L_2(\Omega))$.
If $f=\chi=0$, $\varphi=\psi=0$ then $u=v=0$. 
\item{\rm (ii)} Let  \eqref{discr}, \eqref{paramineq} hold, $p\in ({1\over\alpha},\infty)$, 
\beq\label{phiass}
&&\varphi,\psi\in {\mathcal D}^{r}\;\;\mbox{for some $r>1-{1\over p\alpha}$},
\eeq
$f(t,\cdot)=\chi(t,\cdot)=0$, $t>t_0$, 
for some $t_0>0$ and 
\beq
 \label{fass2}&&\hskip-1truecm \left.\begin{split}
&f|_{(0,t_0)\times\Omega},\chi|_{(0,t_0)\times\Omega}\in L_\infty((0,t_0); {\mathcal D}^{r_1}),
\\
&\sum_{k=1}^\infty\sum_{l=1}^{\mu_k} \lambda_k^{2r_1}\|f_{kl}\|_{L_\infty(0,t_0)}^2<\infty,\; \sum_{k=1}^\infty\sum_{l=1}^{\mu_k}  \lambda_k^{2r_1}\|\chi_{kl}\|_{L_\infty(0,t_0)}^2<\infty
\\
&\mbox{for some  $r_1>0$}.
\end{split}\right\}
 \eeq
Then \eqref{e1} - \eqref{e6} has a solution $(u,v)\in (\mathcal{W}_{\alpha,p})^2$. Its Fourier coefficients satisfy the estimate
\beq\label{thm1k}
\begin{split}
&|u_{kl}(t)|\!+\!|v_{kl}(t)|\!
\le\! c_0\big\{|\varphi_{kl}|\!+\!|\psi_{kl}|\!+t^{\alpha-1}\!*\!(|f_{kl}|\!+\!|\chi_{kl}|)(t)\!\big\},\, 
\\
&t>0, k\in\N,\, 
l=1,\ldots,\mu_k,
\end{split}
\eeq
where $c_0$ is a constant independent of $t$ and $k$. 
 Moreover,
\beq\label{uexp}
&&\hskip-1.4truecm e^{-\sigma t}u,
e^{-\sigma t}v\in L_1((0,\infty);H_0^2(\Omega)),\; \sigma>0,
\eeq
and $(u(t,\cdot),v(t,\cdot))$, $t>t_0$, can be extended to $(u(z,\cdot),v(z,\cdot))$, $z\in \Sigma (t_0,{(2-\alpha)\pi\over 2\alpha})$, 
that is analytic as a function with values in $(H_0^2(\Omega))^2$.
\end{description}
\end{theorem}

Proof of Theorem \ref{dpthm1} is contained in Appendix of the paper.

\medskip
\noindent {\bf Remark 1} By means of embedding theorems it is possible to show that the assumptions
\eqref{fass2} are satisfied if
 $f|_{(0,t_0)\times\Omega},\chi|_{(0,t_0)\times\Omega}\in W_2^1((0,t_0); {\mathcal D}^{r_1})$ for some $r_1>0$.

\medskip
In case $a=0$ we can deduce from Theorem \ref{dpthm1} a uniqueness and existence result for the independent problem  \eqref{e1}, \eqref{e3}, \eqref{e5} for $u$. Such a statement
is formulated below and it follows if 
 we set e.g. $\chi=0$, $\psi=0$, $\varkappa=1$, $b=d=0$ in \eqref{e2}, \eqref{e4}, \eqref{e6}.

\begin{corollary}\label{cor1} Assume that $a=0$. 
\item{\rm (i)} Let \eqref{e1}, \eqref{e3}, \eqref{e5} have a solution $u\in \mathcal{W}_{\alpha,p}$ for some $p\in (1,\infty)$ 
and there exist $\sigma\in\R$ such that $e^{-\sigma t}u \in L_1((0,\infty);L_2(\Omega))$.
If $f=0$, $\varphi=0$ then $u=0$. 
\item{\rm (ii)} Let $p\in ({1\over\alpha},\infty)$, 
\beq\label{cor11}
\varphi\in {\mathcal D}^{r}\;\;\mbox{for some $r>1-{1\over p\alpha}$},
\eeq
$f(t,\cdot)=0$, $t>t_0$, 
for some $t_0>0$ and 
\beq
 \label{cor12}&&\hskip-1truecm \left.\begin{split}
&f|_{(0,t_0)\times\Omega}\in L_\infty((0,t_0); {\mathcal D}^{r_1}),
\;\sum_{k=1}^\infty \lambda_k^{2r_1}\|f_k\|_{L_\infty(0,t_0)}^2<\infty
\\
&\mbox{for some  $r_1>0$}.\end{split}\right\}
 \eeq
Then \eqref{e1}, \eqref{e3}, \eqref{e5} has a solution $u\in \mathcal{W}_{\alpha,p}$.  Moreover,
$$e^{-\sigma t}u\in L_1((0,\infty);H_0^2(\Omega)),\; \sigma>0,$$ 
and $u(t,\cdot)$, $t>t_0$, can be extended to $u(z,\cdot)$, $z\in \Sigma (t_0,{(2-\alpha)\pi\over 2\alpha})$, 
that is analytic as a function with values in $H_0^2(\Omega)$.
\end{corollary}

\section{Uniqueness for inverse problems}\label{secinv}

\subsection{Preliminary results}

Let us denote by $\Phi$ the trace operator that assigns to a function defined in $\Omega$ its normal derivative defined in $\Gamma$, i.e.
\beq\label{Phidef}
(\Phi w)(x')={\partial\over\partial\nu}w(x'),\; x'\in\Gamma.
\eeq
Moreover, we denote $\gamma_{kl}=\Phi{\rm z}_{kl}$. It is known that 
\beq\label{linind}
\forall k\in\N \;\;\mbox{the vectors $\gamma_{kl}$, $l=1,\ldots,\mu_k$, are linearly independent}.
\eeq
This can be proved by means of the elliptic continuation argument (see e.g. \cite{Kianetal}). In the one-dimensional case \eqref{linind} also follows by a simple 
computation: $\gamma_{k1}=-{d\over dx}\sqrt{2\over \pi}\sin (kx)|_{x=0}=-\sqrt{2\over \pi}k\ne 0$, $k\in\N$.

\begin{lemma}\label{iplem0}
Let   $p\in ({1\over\alpha},\infty)$ and $a,b,c,d$, $\varphi$, $\psi$ satisfy \eqref{discr}, \eqref{paramineq}, \eqref{phiass}. 
Let 
 $f(t,\cdot)=\chi(t,\cdot)=0$, $t>t_0$, 
for some $t_0>0$, the relations \eqref{fass2} be valid and 
$(u,v)\in (\mathcal{W}_{\alpha,p})^2$ be the solution \eqref{e1} - \eqref{e6}. Moreover, let us define
\beq\label{Ydef}
Y=\C\;\;\mbox{in case $N=1$},\quad Y=L_2(\Gamma)\;\;\mbox{in case $N\ge 2$}
\eeq
  and assume
 \beq\label{iplem01}  &&\sum_{k=1}^\infty\sum_{l=1}^{\mu_k}\|\gamma_{kl}\|_Y  |\varphi_{kl}|<\infty,\; \sum_{k=1}^\infty\sum_{l=1}^{\mu_k}\|\gamma_{kl}\|_Y \|f_{kl}\|_{L_1(0,t_0)} <\infty,
 \\ \label{iplem02}
  &&\sum_{k=1}^\infty\sum_{l=1}^{\mu_k} \|\gamma_{kl}\|_Y |\psi_{kl}|<\infty,\;
 \sum_{k=1}^\infty\sum_{l=1}^{\mu_k} \|\gamma_{kl}\|_Y \|\chi_{kl}\|_{L_1(0,t_0)}<\infty.
 \eeq
 Then  $e^{-\sigma t}\Phi u\in L_1((0,\infty);Y)$, $\sigma >0$, and
 \beq\label{iplem03}
 \begin{split}
&({\mathcal L}\Phi u(t,\cdot))(s)
\\
&=
\sum_{k=1}^\infty\sum_{l=1}^{\mu_k} {(s^\alpha\!+\!\varkappa\lambda_k\!+\!d)(F_{kl}(s)\!+\!s^{\alpha-1}\varphi_{kl})\!-\!a({\mathcal X}_{kl}(s)\!+\!s^{\alpha-1}\psi_{kl})\over (s^\alpha+\breve\lambda_k)(s^\alpha+\hat\lambda_k)}\gamma_{kl}
\end{split}
\eeq
for ${\rm Re}s>0$, where $\breve\lambda_{k}$ and $\hat\lambda_{k}$ are given by  \eqref{lambdad}. The series on the right-hand side of 
\eqref{iplem03} can be analytically extended to $\Sigma(0,\pi)$,
 has finite one-sided limits at the negative part of real axis and these limits can be moved under the series. 
\end{lemma}

\noindent{\it Proof}. In view of \eqref{uexp} and $\Phi\in {\mathcal B}(H_0^2(\Omega);Y)$, we have  $e^{-\sigma t}\Phi u\in L_1((0,\infty);Y)$, $\sigma >0$, and
\beq\label{lem0t2}
({\mathcal L}\Phi u(t,\cdot))(s)\!=\!\int_0^\infty\!\! e^{-st}\Phi\sum_{k=1}^\infty\sum_{l=1}^{\mu_l} u_{kl}(t){\rm z}_{kl} dt\!=\!\int_0^\infty\!\! e^{-st}\sum_{k=1}^\infty\sum_{l=1}^{\mu_l}  u_{kl}(t)\gamma_{kl} dt
\eeq
for ${\rm Re}s>0$. On the other hand, using \eqref{thm1k}, \eqref{iplem01}, \eqref{iplem02}  we deduce
\beq\nonumber
&&\hskip-7truemm \sum_{k=1}^\infty\sum_{l=1}^{\mu_k}\int_0^\infty \! |e^{-st}  u_{kl}(t)| \|\gamma_{kl}\|_Ydt\le c_0\sum_{k=1}^\infty\sum_{l=1}^{\mu_{k}} \|\gamma_{kl}\|_Y
\int_0^\infty\!\! e^{-{\rm Re}s t}
\big[|\varphi_{kl}|+|\psi_{kl}|
\\ \nonumber
&&\hskip-7truemm+t^{\alpha-1}*(|f_{kl}|+|\chi_{kl}|)(t)\big]dt
=c_0 \int_0^\infty\!\! e^{-{\rm Re}s t}dt \sum_{k=1}^\infty\sum_{l=1}^{\mu_k} \|\gamma_{kl}\|_Y (|\varphi_{kl}|+|\psi_{kl}|)
\\ \label{lem0t3}
&&\hskip-7truemm+
c_0  \int_0^\infty\!\! e^{-{\rm Re}s t}t^{\alpha-1}dt \sum_{k=1}^\infty\sum_{l=1}^{\mu_k} \|\gamma_{kl}\|_Y  \int_0^{t_0}\!\! e^{-{\rm Re}s t}(|f_{kl}(t)|+|\chi_{kl}(t)|)dt<\infty
\eeq
for ${\rm Re}s>0$.
Therefore, due to Tonelli's theorem 
we can change the order of summation and integration on the right-hand side of \eqref{lem0t2}. We obtain
$({\mathcal L}\Phi u(t,\cdot))(s)=\sum\limits_{k=1}^\infty\sum\limits_{l=1}^{\mu_k}  U_{kl}(s)\gamma_{kl}$, ${\rm Re}s>0$. Using  \eqref{prop1}
and  \eqref{factor} we reach the formula 
 \eqref{iplem03} for ${\rm Re}s>0$. 
 
The relations $f_{kl}(t)=0$, $\chi_{kl}(t)=0$, $t>t_0$, imply that $F_{kl}(s)$, ${\mathcal X}_{kl}(s)$, are entire. Moreover, 
from \eqref{lamineq} and $0<\alpha<1$ it follows that  
\beq\label{lem0t5}
|s^\alpha+\breve\lambda_k|, |s^\alpha+\hat\lambda_k|\ge c_1\lambda_k \sin(\alpha\pi)\ge c_1\lambda_1 \sin(\alpha\pi)>0,\; s\in\C,\; k\in\N.
\eeq
Therefore, all addends under the series 
in \eqref{iplem03} are analytic in $\Sigma (0,\pi)$ and have finite one-sided limits at the negative part of the real axis. Observing \eqref{lem0t5}, the inequalities  
\beq\label{lem0t4}
\begin{split}
&|F_{kl}(s)|\!\le\! \max\{1;e^{-{\rm Re}s t_0}\}\|f_{kl}\|_{L_1(0,t_0)},\, s\in\C
\\
&|{\mathcal X}_{kl}(s)|\!\le\! \max\{1;e^{-{\rm Re}s t_0}\}\|\chi_{kl}\|_{L_1(0,t_0)},\; s\in\C,
\end{split}
\eeq
and \eqref{iplem01}, \eqref{iplem02} we see that in every compact subset $D$ of $\C$,  the series \eqref{iplem03} can  be estimated by a converging series that 
is independent of $s\in D$. Therefore, \eqref{iplem03} converges uniformly in $D$. Consequently, it defines an analytic function in  
$\Sigma(0,\pi)$. Due to the dominated convergence theorem,  it has finite one-sided limits at the negative part of the real axis and these limits can be moved under the series.
 \hfill $\Box$

\medskip
\noindent {\bf Remark 2}.
Sufficient conditions for  \eqref{iplem01}, \eqref{iplem02} are $\varphi,\psi\in {\mathcal D}^n$, $f|_{(0,t_0)\times\Omega}
 ,\break\chi|_{(0,t_0)\times\Omega}\in L_1((0,t_0);{\mathcal D}^n)$, where $n$ is an integer greater than ${N+2\over 2}$.  
Let us show this for $\varphi$. Assuming $\varphi\in {\mathcal D}^n$ we have 
\beqst\nonumber
&&\varphi_{kl}=\int_\Omega \varphi {\rm z}_{kl} dx=-{1\over\lambda_k}\int_\Omega \varphi \Delta{\rm z}_{kl} dx=-{1\over\lambda_k}\int_\Omega \Delta\varphi\, {\rm z}_{kl} dx
\\ \label{rem21}
&&=\ldots=(-1)^n {1\over\lambda_k^n}\int_\Omega (\Delta)^n\varphi\, {\rm z}_{kl} dx\quad\Rightarrow\quad |\varphi_{kl}|\le {1\over \lambda_k^n}\|\varphi\|_{{\mathcal D}^n}.
\eeqst
To estimate the eigenvalues from below, 
we introduce the following sequences:
$\dot\lambda_m=\lambda_k$, $\dot{\rm z}_m={\rm z}_{kl}$, $m=\sum\limits_{j=1}^{k-1}\mu_j+l$.
It is well-known that the number of eigenvalues counting multiplicities  of Dirichlet Laplacian less or equal to $\lambda$ is equivalent to $c\lambda^{N/2}$ as $\lambda\to\infty$, where $c$ is a 
constant. This implies that there exists a constant $\dot c$ such that $\dot\lambda_m\ge \dot c m^{2\over N}$. Summing up, in case $n>{N+2\over 2}$ we obtain
\beqst
&&\sum_{k=1}^\infty\sum_{l=1}^{\mu_k}\|\gamma_{kl}\|_Y|\varphi_{kl}|\le \|\varphi\|_{{\mathcal D}^n}\sum_{k=1}^\infty\sum_{l=1}^{\mu_k}{\|\gamma_{kl}\|_Y\over \lambda_k^n}
=\|\varphi\|_{{\mathcal D}^n}\sum_{m=1}^\infty{\|\Phi \dot {\rm z}_m\|_Y\over \dot\lambda_m^n}
\\
&&\le {\rm const} \sum_{m=1}^\infty{\|\dot {\rm z}_m\|_{H_0^2(\Omega)}\over \dot\lambda_m^n}\le {\rm const} \sum_{m=1}^\infty{1\over \dot\lambda_m^{n-1}}
\le {\rm const} \sum_{m=1}^\infty{1\over m^{2(n-1)\over N}}<\infty.
\eeqst

\begin{lemma}\label{iplem1} Let $J\in\N$,  functions
 $G_{k,j}(z)$, $k\in\N$, $j=1,\ldots,J$, be analytic in $\C\setminus \{0\}$ with values in $Y$, where $Y$ is given by \eqref{Ydef}, and functions $R_{k,j}(z)$, $k\in\N$, $j=1,\ldots,J$, be meromorphic in $\C$ with values in $\C$. Let 
the poles of  $R_{k,j}(z)\in\N$, $k\in\N$, $j=1,\ldots,J$, be located on the lines $P=\{z\, :\, {\rm Arg}z=\pm\pi(1-\alpha)\}$ and
the $(z_1,z_2)$-dependent series 
$$
\sum_{k=1}^\infty \sum_{j=1}^J R_{k,j}(z_1)G_{k,j}(z_2)
$$
be uniformly convergent in every compact subset of $(\C\setminus P)\times (\C\setminus  \{0\})$. Consider the functions
$$
Q(n,z)=\sum_{k=1}^\infty \sum_{j=1}^J R_{k,j}(z)G_{k,j}(z^{1\over\alpha}e^{i2\pi n\over\alpha}),\; z\in \C\setminus P,\; n\in \{0\}\cup \N.
$$
If $Q(0,z)$ vanishes for real $z$ in some positive interval then $Q(n,z)=0$, $z\in \C$, $n\in\N$. 
\end{lemma}

\noindent{\it Proof}. Let $q(z)$ be analytic in an open set $S\subset \C\setminus \{0\}$ with values in $\C$. Then, in virtue of the assumptions of the lemma,
$\sum\limits_{k=1}^\infty\sum\limits_{j=1}^J R_{k,j}(z)G_{k,j}(q(z))$ is analytic in $S\setminus P$. Using this statement for $q(z)=z^{1\over\alpha}e^{i2\pi n\over\alpha}$, we obtain that
 $Q(n,z)$ is analytic in $\Sigma(0,\pi)\setminus P$ for any $n\in\N$. 
Since $Q(0,z)$ vanishes for real $z$ in some positive
interval, by analytic continuation we have $Q(0,z)=0$, $z\in \C$. Next, let us define $S=\C\setminus\{0\}\setminus\{z\,:\, {\rm Arg}z=0\}$
 and
$$q(z)=\left\{\begin{array}{ll} z^{1\over\alpha} &\mbox{if ${\rm Arg} z\in (0,\pi]$}\\ 
z^{1\over\alpha}e^{i2\pi\over\alpha} &\mbox{if ${\rm Arg} z\in (-\pi,0)$}.\end{array}\right.$$
Then we have
$$\sum\limits_{k=1}^\infty\sum\limits_{j=1}^J R_{k,j}(z)G_{k,j}(q(z))=\left\{\begin{array}{ll} Q(0,z) &\mbox{if ${\rm Arg} z\in (0,\pi]\setminus \{-\pi(\alpha-1)\}$},\\ 
Q(1,z) &\mbox{if ${\rm Arg} z\in (-\pi,0)\setminus \{\pi(\alpha-1)\}.$}\end{array}\right.$$
Function $q(z)$ is analytic in $S$. This implies that $\sum\limits_{k=1}^\infty\sum\limits_{j=1}^J R_{k,j}(z)G_{k,j}(q(z))$
is analytic in $S\setminus P$. By analytic continuation, the relation $Q(0,z)= 0$, ${\rm Arg} z\in (0,\pi]$, extends to $Q(1,z)=0$, ${\rm Arg} z\in (-\pi,0)$, 
and in turn to $Q(1,z)=0$, $z\in\C$. Continuing 
in a similar way we prove $Q(n,z)=0$, $z\in \C$, $n\in\N$. \hfill $\Box$

\begin{lemma}
\label{iplem2}
Assume that $\alpha$ is irrational. Let $y\in [0,2\pi)$ and $\epsilon>0$. Then there exists $n\in\N$ such that $\left|e^{{i2\pi n\over\alpha}}-e^{iy}\right|<\epsilon$. 
\end{lemma}
 
\noindent{\it Proof}. 
Let $\lfloor\cdot\rfloor$ denote the floor function. The fact that $\alpha$ is irrational implies that the sequence 
${n\over\alpha}-\lfloor{n\over\alpha}\rfloor$, $n\in\N$, is dense in $[0,1)$ \cite{Staib}. Therefore, the sequence 
$2\pi({n\over\alpha}-\lfloor{n\over\alpha}\rfloor)$, $n\in\N$, is dense in $[0,2\pi)$. Let  $y\in [0,2\pi)$ and $\epsilon>0$. Then there exists $n\in\N$ such that 
$|2\pi({n\over\alpha}-\lfloor{n\over\alpha}\rfloor)-y|<\epsilon$. This yields
\beqst
&&\left|e^{{i2\pi n\over\alpha}}-e^{iy}\right|=\left|e^{i2\pi \left(\lfloor {n\over\alpha}\rfloor + {n\over\alpha}-\lfloor {n\over\alpha}\rfloor \right)}-e^{iy}\right|=
\left|e^{i2\pi \left({n\over\alpha}-\lfloor {n\over\alpha}\rfloor \right)}-e^{iy}\right|
\\
&&=\left|\int_y^{2\pi \left({n\over\alpha}-\lfloor {n\over\alpha}\rfloor \right)}ie^{i\xi}d\xi\right|\le 
\left|2\pi \left({n\over\alpha}-\left\lfloor {n\over\alpha}\right\rfloor \right)-y\right|<\epsilon.
\eeqst
  \hfill $\Box$
  
\subsection{Uniqueness for IP1}

\begin{theorem}\label{ip1thm}
Let $\alpha$ be irrational, $a=0$,
  $p\in ({1\over\alpha},\infty)$ and $\varphi$ satisfy \eqref{cor11}.  Let $f(t,\cdot)=0$, $t>t_0$, 
for some $t_0>0$, \eqref{cor12} hold and 
$u\in \mathcal{W}_{\alpha,p}$ be the solution \eqref{e1}, \eqref{e3}, \eqref{e5}.  Moreover, let 
$\gamma_{kl}=\Phi {\rm z}_{kl}$, where $\Phi$ is given by \eqref{Phidef}, and
\eqref{iplem01} be valid, where $Y$ is defined in \eqref{Ydef}.
 If there exists $t_1>t_0$ such that 
 \beq\label{ip1thm2}
 \Phi u(t,\cdot)=0, \; t\in (t_0,t_1),
 \eeq
 then $f=0$, $\varphi=0$ and $u=0$. 
\end{theorem}

\noindent{\it Proof}. By Lemma \ref{iplem0}, the assumption $a=0$ and \eqref{lambdad} we have the formula
\beq\label{ip1t4}
({\mathcal L}\Phi u(t,\cdot))(s)=
\sum_{k=1}^\infty\sum_{l=1}^{\mu_k} {(F_{kl}(s)+s^{\alpha-1}\varphi_{kl})\over s^\alpha+\kappa\lambda_k+c}\gamma_{kl}
\eeq
in the space $Y$ and the series on the right-hand side of \eqref{ip1t4} analytically extends to all $s\in\Sigma(0,\pi)$.
On the other hand, due $\Phi\in{\mathcal B}(H_0^2(\Omega),Y)$,  and
Corollary \ref{cor1}, $\Phi u(z,\cdot)$ is analytic in $\Sigma (t_0,{(2-\alpha)\pi\over 2\alpha})$. Using \eqref{ip1thm2} and analytic continuation we obtain 
$\Phi u(t,\cdot)=0$, $t>t_0$. This implies that $({\mathcal L}\Phi u(t,\cdot))(s)$ is  entire, hence
continuous at 
 the negative part of the real axis. Let us set $s=re^{i\theta}$, $r>0$, in \eqref{ip1t4},  compute the difference of 
limits in the processes $\theta\to\pm\pi$ and move the limits under the series (the latter step is possible due to Lemma \ref{iplem0}). We obtain
 \beqst
 0=\sum_{k=1}^\infty\sum_{l=1}^{\mu_l} \left\{{F_{kl}(-r)-r^{\alpha-1}e^{i\pi\alpha}\varphi_{kl}\over r^\alpha e^{i\pi\alpha}+\kappa\lambda_k+c}-
 {F_{kl}(-r)-r^{\alpha-1}e^{-i\pi\alpha}\varphi_{kl}\over r^\alpha e^{-i\pi\alpha}+\kappa\lambda_k+c}\right\}\gamma_{kl}.
 \eeqst
Denoting $\varrho=r^\alpha$ we have
 \beq\label{ip1t5}
 0=\sum_{k=1}^\infty\sum_{l=1}^{\mu_k} \left\{{F_{kl}(-\varrho^{1\over\alpha})\!-\!\varrho e^{i\pi\alpha}\varrho^{-{1\over\alpha}}\varphi_{kl}\over \varrho e^{i\pi\alpha}+\kappa\lambda_k+c}-
 {F_{kl}(-\varrho^{1\over\alpha})\!-\!\varrho e^{-i\pi\alpha}\varrho^{-{1\over\alpha}}\varphi_{kl}\over \varrho e^{-i\pi\alpha}+\kappa\lambda_k+c}\right\}\gamma_{kl}
 \eeq
 for $\varrho>0$.
Next, let us define  
\beq\label{ip1t6}
&&G_{k,1}(z)= \sum_{l=1}^{\mu_k}F_{kl}(-z)\gamma_{kl},\; G_{k,2}(z)=-{1\over z}\sum_{l=1}^{\mu_k}\varphi_{kl}\gamma_{kl},\\
\label{ip1t7}
&&\begin{split}
&R_{k,1}(z)={1\over ze^{i\pi\alpha}+\kappa\lambda_k+c}-
{1\over ze^{-i\pi\alpha}+\kappa\lambda_k+c},\;
\\
& R_{k,2}(z)={ze^{i\pi\alpha}\over ze^{i\pi\alpha}+\kappa\lambda_k+c}-
{ze^{-i\pi\alpha}\over ze^{-i\pi\alpha}+\kappa\lambda_k+c}.\end{split}
\eeq
Since $f_{kl}(t)=0$ for $t>t_0$, the functions $F_{kl}(z)$ are entire. Therefore, $G_{k,j}(z)$, $k\in\N$, $j=1,2$, are analytic in $\C\setminus \{0\}$. 
The poles  of meromorphic functions $R_{k,j}(z)$, $k\in\N$, $j=1,2$, are located in $P=\{z\, :\, {\rm Arg}z=\pm\pi(1-\alpha)\}$.
Moreover, the series $\sum\limits_{k=1}^\infty\sum\limits_{j=1}^2 R_{k,j}(z_1)G_{k,j}(z_2)$ 
is uniformly convergent in every compact subset of $(\C\setminus P)\times (\C\setminus  \{0\})$. This follows from  \eqref{iplem01} and the first relation in \eqref{lem0t4}.
Let us define $Q(n,z)= \sum\limits_{k=1}^\infty\sum\limits_{j=1}^2 R_{k,j}(z)G_{k,j}(z^{1\over\alpha}e^{i2\pi n\over\alpha})$, $n\in\N$. 
The right-hand side of \eqref{ip1t5} coincides with 
$Q(0,\varrho)$ for real $\varrho>0$. Consequently, due to \eqref{ip1t5} and
Lemma \ref{iplem1}  we have
\beq\label{ip1t8}
\sum\limits_{k=1}^\infty\sum\limits_{j=1}^2 R_{k,j}(z)G_{k,j}(z^{1\over\alpha}e^{i2\pi n\over\alpha})=0,\;\; z\in \C,\; n\in\N.
\eeq

{\it Now we have reached a point where we can introduce an additional continuous variable to the problem}. Let $z$ be some number in the set $\C\setminus (P\cup \{0\})$
and consider the series 
\beq\label{ip1thmserr}
\sum\limits_{k=1}^\infty\sum\limits_{j=1}^2 R_{k,j}(z)G_{k,j}^{(m)}(z^{1\over\alpha}e^{iy}),\; \;\mbox{where $y\in [0,2\pi]$ and $m\in\{0\}\cup\N$.}
\eeq 
Observing \eqref{ip1t6}, \eqref{ip1t7}, \eqref{iplem01} and the estimate
$$
|F_{kl}^{(m)}(s)|=\left|{d^m\over ds^m} \int_0^{t_0}e^{-st}f_{kl}(t)dt\right|\le t_0^m \max\{1;e^{-{\rm Re}s t_0}\}\|f_{kl}\|_{L_1(0,t_0)},
$$
$s\in\C$, $m\in\{0\}\cup\N$, we see that for any $m\in\{0\}\cup\N$, the series \eqref{ip1thmserr} can be estimated by a converging series that 
is independent of $y\in [0,2\pi]$. 
Therefore, by the dominated convergence theorem, it is possible to move limits and derivatives with respect to $y$  under the series \eqref{ip1thmserr}.

On the other hand, by Lemma \ref{iplem2}, for any $y\in [0,2\pi)$, there exists a sequence of integers $n(y)_q$, $q\in\N$, such that 
$e^{i2\pi n(y)_q\over\alpha}\to e^{iy}$ as $q\to\infty$. Therefore, setting $n=n(y)_q$ in \eqref{ip1t8} and letting $q\to\infty$, we deduce
\beqst
\sum\limits_{k=1}^\infty\sum\limits_{j=1}^2 R_{k,j}(z)G_{k,j}(z^{1\over\alpha}e^{iy})=0,\;\; y\in [0,2\pi).
\eeqst
Applying $m$ times the operator ${1\over z^{1\over\alpha}ie^{iy}}{d\over dy}$ we obtain
\beqst
\sum\limits_{k=1}^\infty\sum\limits_{j=1}^2 R_{k,j}(z)G_{k,j}^{(m)}(z^{1\over\alpha}e^{iy})=0,\;\; y\in [0,2\pi),\; m\in\{0\}\cup\N.
\eeqst
Setting $y=0$ we have
\beq\label{ip1t9}
\sum\limits_{k=1}^\infty\sum\limits_{j=1}^2 R_{k,j}(z)G_{k,j}^{(m)}(z^{1\over\alpha})=0,\;\;  m\in\{0\}\cup\N.
\eeq

Let us choose $n\in\N$ and denote $z_n=-e^{-{i\pi\alpha}}(\kappa\lambda_n+c)$. This is a pole of $R_{n,1}$ and $R_{n,2}$. 
Moreover, 
\beqst
\lim_{z\to z_n}(z-z_n)R_{n,1}(z)=e^{-i\pi\alpha},\; \lim_{z\to z_n}(z-z_n)R_{n,2}(z)|=z_n,
\eeqst
and in view of $\lambda_k\ne \lambda_n$, $k\ne n$, we have
\beqst
(z-z_n)R_{k,j}(z)|_{z=z_n}=0,\; k\ne n,\; j=1,2.
\eeqst
Thus, multiplying \eqref{ip1t9} by $z-z_n$ and passing to the limit $z\to z_n$ we obtain 
\beq\label{ip1t9a}
e^{-i\pi\alpha}G_{n,1}^{(m)}(z_n^{{1\over\alpha}})+z_n G_{n,2}^{(m)}(z_n^{{1\over\alpha}})=0,\; m\in\{0\}\cup\N.
\eeq
%Note that $z_m^{1\over\alpha}=(\kappa\lambda_m+c)e^{i\pi\left({1\over\alpha}-1\right)}$. 
Let us consider the function  $$G_n(z)=e^{-i\pi\alpha}G_{n,1}(z)+z_n G_{n,2}(z).$$ 
This is  analytic in $z\in \C\setminus\{0\}$ with a possible pole at $z=0$. 
The relations \eqref{ip1t9a} show that $G_n^{(m)}(z)$ vanishes at $z=z_n^{{1\over\alpha}}$ for all $m\in\{0\}\cup\N$.  Consequently, $G_n(z)$ vanishes everywhere, i.e.
 $$e^{-i\pi\alpha}G_{n,1}(z)+z_n G_{n,2}(z)=0,\; z\in \C.$$ Using \eqref{ip1t6}  we have
\beq\label{ip1t10}
e^{-i\pi\alpha}\sum_{l=1}^{\mu_n}F_{nl}(-z)\gamma_{nl}-{z_n\over z}\sum_{l=1}^{\mu_n}\varphi_{nl}\gamma_{nl}=0,\;\; z\in \C.
\eeq
Excluding the pole at the left-hand side of \eqref{ip1t10}, we have $\sum\limits_{l=1}^{\mu_n}\varphi_{nl}\gamma_{nl}=0$. 
Due to \eqref{linind} we have $\varphi_{nl}=0$, $l=1,\ldots,\mu_n$. 
 Now \eqref{ip1t10} reduces to
$\sum\limits_{l=1}^{\mu_n}F_{nl}(-z)\gamma_{nl}=0$, $z\in\C$. Using again \eqref{linind} we obtain $F_{nl}(z)=0$, $z\in\C$, $l=1,\ldots,\mu_n$. 
 This yields $f_{nl}=0$, $l=1,\ldots,\mu_n$.  Since $n\in\N$ is arbitrary, we have
 $\varphi=0$, $f=0$. 
 Finally, Corollary \ref{cor1} (i) implies $u=0$. \hfill $\Box$

\subsection{Uniqueness for IP2}

\begin{theorem}\label{ip2thm}
Let $\alpha$ be irrational, $a\ne 0$,  $p\in ({1\over\alpha},\infty)$ and $a,b,c,d$, $\varphi$, $\psi$ satisfy \eqref{discr}, \eqref{paramineq}, \eqref{phiass}. 
Let 
 $f(t,\cdot)=\chi(t,\cdot)=0$, $t>t_0$, 
for some $t_0>0$, the relations \eqref{fass2} be valid and 
$(u,v)\in (\mathcal{W}_{\alpha,p})^2$ be the solution \eqref{e1} - \eqref{e6}.
Moreover, let 
$\gamma_{kl}=\Phi {\rm z}_{kl}$, where $\Phi$ is given by \eqref{Phidef}, and
\eqref{iplem01}, \eqref{iplem02} be valid, where $Y$ is defined in \eqref{Ydef}.
Assume that
\beq\label{ip2thm1}
\breve\lambda_k\ne \breve\lambda_n,\, \hat\lambda_k\ne \hat\lambda_n,\, \hat\lambda_k\ne \breve\lambda_n\;\;\mbox{for}\;\; k,n\in\N,\; k\ne n,
\eeq
where $\breve\lambda_k$, $\hat\lambda_k$  are given by \eqref{lambdad}. 
 If there exists $t_1>t_0$ such that \eqref{ip1thm2} holds 
  then $f=\chi=0$, $\varphi=\psi=0$ and $u=v=0$. 
\end{theorem}

\noindent{\it Proof}. By Lemma \ref{iplem0} the formula \eqref{iplem03} is valid for $({\mathcal L}\Phi u(t,\cdot))(s)$ in the space $Y$.
On the other hand, due to $\Phi\in{\mathcal B}(H_0^2(\Omega),Y)$ and Theorem \ref{dpthm1}, 
 $\Phi u(z,\cdot)$ is analytic in $\Sigma (t_0,{(2-\alpha)\pi\over 2\alpha})$. Therefore, from \eqref{ip1thm2} we obtain 
$\Phi u(t,\cdot)=0$, $t>t_0$. This implies that $({\mathcal L}\Phi u(t,\cdot))(s)$ is  entire, hence continuous at the negative part 
of the real axis. Let us set  $s=re^{i\theta}$, $r>0$, in \eqref{iplem03}, compute the difference of 
limits    in the processes $\theta\to\pm\pi$ and denote $\varrho=r^\alpha$. We obtain
\beq\nonumber
&&\hskip-.7truecm 0=\sum_{k=1}^\infty\sum_{l=1}^{\mu_k}\Bigg\{{1\over (\varrho e^{i\pi\alpha}\!+\!\breve\lambda_k)(\varrho e^{i\pi\alpha}\!+\!\hat\lambda_k)}
\\ \nonumber
&&\hskip-.0truecm\times \left[ (\varrho e^{i\pi\alpha}\!+\!\varkappa\lambda_k\!+\!d)\left(F_{kl}(-\varrho^{1\over\alpha})\!-\!\varrho e^{i\pi\alpha}\varrho^{-{1\over\alpha}}\varphi_{kl}\right)
\right.
\\ \nonumber
&&\hskip-.0truecm\left.-a\left({\mathcal X}_{kl}(-\varrho^{1\over\alpha})\!-\!\varrho e^{i\pi\alpha}\varrho^{-{1\over\alpha}}\psi_{kl}\right)\right] 
\\  \nonumber
&&\hskip-.0truecm -{1\over (\varrho e^{-i\pi\alpha}\!+\!\breve\lambda_k)(\varrho e^{-i\pi\alpha}\!+\!\hat\lambda_k)}
\\ \nonumber
&&\hskip-0truemm \times\left[ (\varrho e^{-i\pi\alpha}\!+\!\varkappa\lambda_k\!+\!d)\left(F_{kl}(-\varrho^{1\over\alpha})\!-
\!\varrho e^{-i\pi\alpha}\varrho^{-{1\over\alpha}}\varphi_{kl}\right)\right.
\\ \label{ip2t5}
&&\hskip-0truemm\left.
-a\left({\mathcal X}_{kl}(-\varrho^{1\over\alpha})\!-\!\varrho e^{-i\pi\alpha}\varrho^{-{1\over\alpha}}\psi_{kl}\right)\right]\Bigg\}\gamma_{kl}
\eeq
for $\varrho>0$. Let us define the following functions:
\beq\label{ip2t6}
&&\hskip-12truemm \left.\begin{split}
&G_{k,1}(z)= \sum_{l=1}^{\mu_k}F_{kl}(-z)\gamma_{kl},\; G_{k,2}(z)=-{1\over z}\sum_{l=1}^{\mu_k}\varphi_{kl}\gamma_{kl},\;
\\
&G_{k,3}(z)= \sum_{l=1}^{\mu_k}{\mathcal X}_{kl}(-z)\gamma_{kl},\; G_{k,4}(z)=-{1\over z}\sum_{l=1}^{\mu_k}\psi_{kl}\gamma_{kl},\;
\end{split}\right\}
\\
\label{ip2t7}
&&\hskip-12truemm\left.\begin{split}
&R_{k,1}(z)={z e^{i\pi\alpha}+\varkappa\lambda_k+d\over (ze^{i\pi\alpha}+\breve\lambda_k)(ze^{i\pi\alpha}+\hat\lambda_k)}-
{z e^{-i\pi\alpha}+\varkappa\lambda_k+d\over (ze^{-i\pi\alpha}+\breve\lambda_k)(ze^{-i\pi\alpha}+\hat\lambda_k)},
\\
& R_{k,2}(z)={ze^{i\pi\alpha}(z e^{i\pi\alpha}+\varkappa\lambda_k+d)\over (ze^{i\pi\alpha}+\breve\lambda_k)(ze^{i\pi\alpha}+\hat\lambda_k)}-
{ze^{-i\pi\alpha}(z e^{-i\pi\alpha}+\varkappa\lambda_k+d)\over (ze^{-i\pi\alpha}+\breve\lambda_k)(ze^{-i\pi\alpha}+\hat\lambda_k)},
\\[1ex]&R_{k,3}(z)=-{a\over (ze^{i\pi\alpha}+\breve\lambda_k)(ze^{i\pi\alpha}+\hat\lambda_k)}+
{a\over (ze^{-i\pi\alpha}+\breve\lambda_k)(ze^{-i\pi\alpha}+\hat\lambda_k)},
\\
& R_{k,4}(z)=-{aze^{i\pi\alpha}\over (ze^{i\pi\alpha}+\breve\lambda_k)(ze^{i\pi\alpha}+\hat\lambda_k)}+
{aze^{-i\pi\alpha}\over (ze^{-i\pi\alpha}+\breve\lambda_k)(ze^{-i\pi\alpha}+\hat\lambda_k)}.
\end{split}\right\}
\eeq
Since $f_{kl}(t)=0$, $\chi_{kl}(t)=0$ for $t>t_0$, the functions $F_{kl}(z)$ and ${\mathcal X}_{kl}(z)$ are entire. Hence, $G_{k,j}(z)$ are analytic in $\C\setminus \{0\}$. 
The poles  of $R_{k,j}(z)$, $k\in\N$, $j=1,\ldots,4$, are contained in $P=\{z\, :\, {\rm Arg}z=\pm\pi(1-\alpha)\}$.
In addition, the series $\sum\limits_{k=1}^\infty\sum\limits_{j=1}^4 R_{k,j}(z_1)G_{k,j}(z_2)$ 
is uniformly convergent in every compact subset of $(\C\setminus P)\times (\C\setminus  \{0\})$. This follows from  \eqref{iplem01}, \eqref{iplem02} and  \eqref{lem0t4}.
Define $Q(n,z)=\sum\limits_{k=1}^\infty\sum\limits_{j=1}^4 R_{k,j}(z)G_{k,j}(z^{1\over\alpha}e^{i2\pi n\over\alpha})$, $n\in\N$. 
The right-hand side of \eqref{ip2t5} coincides with 
$Q(0,\varrho)$ for real $\varrho>0$.
Consequently, by \eqref{ip2t5} and Lemma \ref{iplem1} we have 
\beq\label{ip2t8}
\sum\limits_{k=1}^\infty\sum\limits_{j=1}^4 R_{k,j}(z)G_{k,j}(z^{1\over\alpha}e^{i2\pi n\over\alpha})=0, \; z\in\C,\; n\in\N.
\eeq

As in the proof of Theorem \ref{ip1thm}, for any $y\in [0,2\pi)$, we choose a sequence of integers $n(y)_q$, $q\in\N$, such that 
$e^{i2\pi n(y)_q\over\alpha}\to e^{iy}$ as $q\to\infty$. Setting $n=n(y)_q$  in \eqref{ip2t8}   and passing to the limit $q\to\infty$ we obtain 
$\sum\limits_{k=1}^\infty\sum\limits_{j=1}^4 R_{k,j}(z)G_{k,j}(z^{1\over\alpha}e^{iy})=0$, $y\in [0,2\pi)$.
Applying $m$ times the operator ${1\over z^{1\over\alpha}ie^{iy}}{d\over dy}$ and setting $y=0$ we deduce
\beq\label{ip2t9}
\sum\limits_{k=1}^\infty\sum\limits_{j=1}^4 R_{k,j}(z)G_{k,j}^{(m)}(z^{1\over\alpha})=0,\;\;  m\in\{0\}\cup \N.
\eeq

Let $n$ be an arbitrary integer in $\N$. Our next aim is to show that
\beq\label{ip2tmm}
\varphi_{nl}=\psi_{nl}=0,\;\; F_{nl}(z)={\mathcal X}_{nl}(z)=0,\; z\in\C,\; l=1\ldots,\mu_n.
\eeq
Firstly, we consider the case $\hat\lambda_n=\breve\lambda_n$. Then
\beq\label{ip2t90}&&\left.
\begin{split}
&\hskip-5truemm R_{n,1}(z)={z e^{i\pi\alpha}+\varkappa\lambda_n+d\over (ze^{i\pi\alpha}+\breve\lambda_n)^2}-
{z e^{-i\pi\alpha}+\varkappa\lambda_n+d\over (ze^{-i\pi\alpha}+\breve\lambda_n)^2},
\\
&\hskip-5truemm R_{n,2}(z)={z e^{i\pi\alpha}(z e^{i\pi\alpha}+\varkappa\lambda_n+d)\over (ze^{i\pi\alpha}+\breve\lambda_n)^2}-
{z e^{-i\pi\alpha}(z e^{-i\pi\alpha}+\varkappa\lambda_n+d)\over (ze^{-i\pi\alpha}+\breve\lambda_n)^2},
\\[1ex]
&\hskip-5truemm R_{n,3}(z)=-{a\over (ze^{i\pi\alpha}+\breve\lambda_n)^2}+
{a\over (ze^{-i\pi\alpha}+\breve\lambda_n)^2},
\\
&\hskip-5truemm R_{n,4}(z)=-{aze^{i\pi\alpha}\over (ze^{i\pi\alpha}+\breve\lambda_n)^2}+
{aze^{-i\pi\alpha}\over (ze^{-i\pi\alpha}+\breve\lambda_n)^2}.\end{split}\right\}
\eeq
Define $z_n=-e^{-i\pi\alpha}\breve\lambda_n$, $n\in\N$, multiply \eqref{ip2t9} by $(z-z_n)^2$ and let $z\to z_n$. 
Observing the equalities
\beq\label{ipthm200}
(z-z_n)R_{k,j}(z)|_{z=z_n}=0,\; k\ne n,\, j=1,\ldots,4,
\eeq
that follow from  \eqref{ip2thm1} we 
obtain
\beqst
&&\hskip-9truemm e^{-i2\pi\alpha}(z_n e^{i\pi\alpha}\!+\!\varkappa\lambda_n\!+\!d)G_{n,1}^{(m)}(z_n^{1\over\alpha})+
z_ne^{-i\pi\alpha}(z_n e^{i\pi\alpha}\!+\!\varkappa\lambda_n\!+\!d)G_{n,2}^{(m)}(z_n^{1\over\alpha})
\\
&&-ae^{-i2\pi\alpha}G_{n,3}^{(m)}(z_n^{1\over\alpha})-az_ne^{-i\pi\alpha}G_{n,4}^{(m)}(z_n^{1\over\alpha})=0,\; m\in\{0\}\cup \N.
\eeqst
Taking the formula of $z_n$  into account we have
\beqst
&&\hskip-9truemm(\varkappa\lambda_n+d-\breve\lambda_n)G_{n,1}^{(m)}(z_n^{1\over\alpha})-\breve\lambda_n(\varkappa\lambda_n+d-\breve\lambda_n)
G_{n,2}^{(m)}(z_n^{1\over\alpha})
\\
&&-aG_{n,3}^{(m)}(z_n^{1\over\alpha})+a\breve\lambda_nG_{n,4}^{(m)}(z_n^{1\over\alpha})=0,\; m\in\{0\}\cup \N.
\eeqst
Since $G_{n,j}(z)$, $j=1,\ldots,4$,  are analytic in $\C\setminus\{0\}$, this implies
\beqst\label{ip2t91}
\begin{split}
&\hskip-9truemm(\varkappa\lambda_n+d-\breve\lambda_n)G_{n,1}(z)-\breve\lambda_n(\varkappa\lambda_n+d-\breve\lambda_n)
G_{n,2}(z)
\\
&-aG_{n,3}(z)+a\breve\lambda_nG_{n,4}(z)=0,\; z\in\C.
\end{split}
\eeqst
By means of \eqref{ip2t6} we have
\beqst
\begin{split}
&\hskip-9truemm\sum_{l=1}^{\mu_n}\Big\{(\varkappa\lambda_n+d-\breve\lambda_n)F_{nl}(-z)-a{\mathcal X}_{nl}(-z)
\\
&+{\breve\lambda_n\over z} \left((\varkappa\lambda_n+d-\breve\lambda_n)\varphi_{nl}-a\psi_{nl}\right)\Big\}\gamma_{nl}=0,\;
z\in \C.
\end{split}
\eeqst
In view of \eqref{linind} we obtain
\beq\label{ip2t95}
\begin{split}
&\hskip-9truemm(\varkappa\lambda_n+d-\breve\lambda_n)F_{nl}(-z)-a{\mathcal X}_{nl}(-z)
\\
&+{\breve\lambda_n\over z} \left((\varkappa\lambda_n+d-\breve\lambda_n)\varphi_{nl}-a\psi_{nl}\right)=0,\;
z\in \C,\; l=1,\ldots,\mu_n.
\end{split}
\eeq
Excluding the pole and observing that  $a\ne 0$ we have
\beq\label{ip2t96}
\psi_{nl}={\varkappa\lambda_n+d-\breve\lambda_n\over a}\varphi_{nl},\; l=1,\ldots,\mu_n.
\eeq
Using this in  \eqref{ip2t95} we get
\beq\label{ip2t97}
{\mathcal X}_{nl}(z)={\varkappa\lambda_n+d-\breve\lambda_n\over a}F_{nl}(z),\; z\in\C,\; l=1,\ldots,\mu_n.
\eeq
Next we will show that $\varphi_{nl}$ and $F_{nl}$ vanish. To this end we will simplify \eqref{ip2t9} and  perform another limit process with it.
In view of \eqref{ip2t96} and \eqref{ip2t97} it holds 
 $G_{n,3}(z)={\varkappa\lambda_n+d-\breve\lambda_n\over a}G_{n,1}(z)$ and $G_{n,4}(z)={\varkappa\lambda_n+d-\breve\lambda_n\over a}G_{n,2}(z)$.
Observing these relations and formulas \eqref{ip2t90}  we see that 
\eqref{ip2t9} reduces to 
\beq\label{ip2t9a}
\sum\limits_{k=1\atop k\ne n}^\infty\sum\limits_{j=1}^4  R_{k,j}(z)G_{k,j}^{(m)}(z^{1\over\alpha})+
\sum_{j=1}^2\widetilde R_{n,j}(z)G_{n,j}^{(m)}(z^{1\over\alpha})=0,\;\;  m\in\{0\}\cup \N,\; 
\eeq
where
\beqst
\begin{split}
&\widetilde R_{n,1}(z)={1\over ze^{i\pi\alpha}+\breve\lambda_n}-
{1\over ze^{-i\pi\alpha}+\breve\lambda_n},
\\
&\widetilde  R_{n,2}(z)={z e^{i\pi\alpha}\over ze^{i\pi\alpha}+\breve\lambda_n}-
{z e^{-i\pi\alpha}\over ze^{-i\pi\alpha}+\breve\lambda_n}.
\end{split}
\eeqst
Multiplying \eqref{ip2t9a} by $z-z_n$, letting $z\to z_n$ and taking \eqref{ipthm200} into account we obtain 
$$
e^{-i\pi\alpha}G_{n,1}^{(m)}(z_n^{1\over\alpha})+z_n G_{n,2}^{(m)}(z_n^{1\over\alpha})=0,\; m\in\{0\}\cup \N.
$$
This by the analyticity of $G_{n,j}$ implies $e^{-i\pi\alpha}G_{n,1}(z)+z_n G_{n,2}(z)=0$, $z\in\C$. Using \eqref{ip2t6} and \eqref{linind} we obtain the relations 
$e^{-i\pi\alpha}F_{nl}(-z)-{z_n\over z}\varphi_{nl}=0$, $z\in \C$, $l=1,\ldots,\mu_n$. These imply $\varphi_{nl}=0$ and $F_{nl}(z)=0$, $z\in\C$, for 
$l=1,\ldots,\mu_n$.  Finally, returning to \eqref{ip2t96} and \eqref{ip2t97} we obtain
$\psi_{nl}=0$ and ${\mathcal X}_{nl}(z)=0$, $z\in\C$, for $l=1,\ldots,\mu_n$. The relations \eqref{ip2tmm} are proved in the case $\hat\lambda_n=\breve\lambda_n$.

Secondly, we consider the case $\hat\lambda_n\ne \breve\lambda_n$. 
As before, we set $z_n=-e^{-i\pi\alpha}\breve\lambda_n$. Multiplying \eqref{ip2t9} by $z-z_n$, taking the limit $z\to z_n$ and observing the relations 
\eqref{ipthm200} we obtain the equality
\beqst 
&&{z_ne^{i\pi\alpha}+\varkappa\lambda_n+d\over  z_ne^{i\pi\alpha}+\hat\lambda_n}\left(e^{-i\pi\alpha}G_{n,1}^{(m)}( z_n^{1\over\alpha})+ z_nG_{n,2}^{(m)}( z_n^{1\over\alpha})\right)
\\
&&-{a\over \ z_ne^{i\pi\alpha}+\hat\lambda_n}\left(e^{-i\pi\alpha}G_{n,3}^{(m)}( z_n^{1\over\alpha})+ z_nG_{n,4}^{(m)}( z_n^{1\over\alpha})\right)=0,\; m\in \{0\}\cup\N.
\eeqst
After simplification we have
\beqst
(\varkappa\lambda_n\!+\!d\!-\!\breve\lambda_n)\left(G_{n,1}^{(m)}( z_n^{1\over\alpha})\!-\!\breve\lambda_nG_{n,2}^{(m)}( z_n^{1\over\alpha})\right)\!-\!
a\left(G_{n,3}^{(m)}( z_n^{1\over\alpha})\!-\!\breve\lambda_nG_{n,4}^{(m)}(z_n^{1\over\alpha})\right)=0
\eeqst
for  $m\in \{0\}\cup\N$.
This by the analyticity argument yields
\beqst
(\varkappa\lambda_n+d-\breve\lambda_n)\left(G_{n,1}(z)-\breve\lambda_nG_{n,2}(z)\right)-
a\left(G_{n,3}(z)-\breve\lambda_nG_{n,4}(z)\right)=0
\eeqst
for $z\in\C$. By means of \eqref{ip2t6} and \eqref{linind} we obtain
\beq\label{ip2t20}
\begin{split}
&(\varkappa\lambda_n+d-\breve\lambda_n)\Big(F_{nl}(-z)+{\breve\lambda_n\over z}\varphi_{nl}\Big)
-
a\Big({\mathcal X}_{nl}(-z)+{\breve\lambda_n\over z}\psi_{nl}\Big)=0,
\\
&\qquad z\in\C,\; l=1,\ldots,\mu_n.
\end{split}
\eeq
Next, we
set $\hat z_n=-e^{-i\pi\alpha}\hat\lambda_n$,  multiply \eqref{ip2t9} by $z-\hat z_n$ and take the limit $z\to\hat z_n$. 
Observing the relations 
\beqst
(z-\hat z_n)R_{k,j}(z)|_{z=\hat z_n}=0,\; k\ne n,\, j=1,\ldots,4,
\eeqst
that also follow from \eqref{ip2thm1}, 
 we after a simplification reach the 
expression
\beqst
(\varkappa\lambda_n\!+\!d\!-\!\hat\lambda_n)\left(G_{n,1}^{(m)}(\hat z_n^{1\over\alpha})\!-\!\hat\lambda_nG_{n,2}^{(m)}(\hat z_n^{1\over\alpha})\right)\!-\!
a\left(G_{n,3}^{(m)}(\hat z_n^{1\over\alpha})\!-\!\hat\lambda_nG_{n,4}^{(m)}(\hat z_n^{1\over\alpha})\right)=0
\eeqst
for $m\in \{0\}\cup \N$. 
This by the analyticity yields $$(\varkappa\lambda_n+d-\hat\lambda_n)\left(G_{n,1}(z)-\hat\lambda_nG_{n,2}(z)\right)-
a\left(G_{n,3}(z)-\hat\lambda_nG_{n,4}(z)\right)=0$$ for $z\in\C$. In view of of \eqref{ip2t6} and \eqref{linind} we get
\beq\label{ip2t21}
\begin{split}
&(\varkappa\lambda_n+d-\hat\lambda_n)\Big(F_{nl}(-z)+{\hat\lambda_n\over z}\varphi_{nl}\Big)
-
a\Big({\mathcal X}_{nl}(-z)+{\hat\lambda_n\over z}\psi_{nl}\Big)=0,
\\
&\qquad z\in\C,\; l=1,\ldots,\mu_n.
\end{split}
\eeq
We continue analyzing the system \eqref{ip2t20}, \eqref{ip2t21}.
Excluding the poles we obtain
the following $2\times 2$ linear systems for the pairs $\varphi_{nl}$, $\psi_{nl}$, where $l=1,\ldots,\mu_n$:
\beqst
(\varkappa\lambda_n+d-\breve\lambda_n)\varphi_{nl}-a\psi_{nl}=0,\;\; (\varkappa\lambda_{n}+d-\hat\lambda_n)\varphi_{nl}-a\psi_{nl}=0.
\eeqst
Since $\breve\lambda_n\ne \hat\lambda_n$ and $a\ne 0$, these systems are regular. 
The unique solutions are $\varphi_{nl}=\psi_{nl}=0$, $l=1\ldots,\mu_n$. Using these equalities in \eqref{ip2t20}, \eqref{ip2t21} we obtain the linear systems
for $F_{nl}(-z)$, ${\mathcal X}_{nl}(-z)$, $z\in\C$ for any $l=1,\ldots,\mu_n$:
\beqst
(\varkappa\lambda_n\!+\!d\!-\!\breve\lambda_n)F_{nl}(-z)\!-\!a{\mathcal X}_{nl}(-z)=0,\;
 (\varkappa\lambda_n\!+\!d\!-\!\hat\lambda_n)F_{nl}(-z)\!-\!a{\mathcal X}_{nl}(-z)\!=\!0.
\eeqst
The unique solutions are $F_{nl}(-z)={\mathcal X}_{nl}(-z)=0$, $z\in\C$, $l=1,\ldots,\mu_n$. We have shown that \eqref{ip2tmm} holds in the case $\breve\lambda_n\ne \hat\lambda_n$, too. 

Summing up, we have proved that \eqref{ip2tmm} is valid for all $n\in\N$. This yields $\varphi=\psi=0$ and $f=\chi=0$. 
 Finally, Theorem \ref{dpthm1} (i) implies $u=v=0$. Theorem is  proved.
\hfill $\Box$

\medskip
Let us interpret the condition
\eqref{ip2thm1} in some particular cases. If $\kappa=\varkappa$ then
\beqst
\begin{split}
&\breve\lambda_k=\kappa\lambda_k+{1\over 2}\left[c+d+\sqrt{(c-d)^2+4ab}\right],
\\
&\hat\lambda_k=\kappa\lambda_k+{1\over 2}\left[c+d-\sqrt{(c-d)^2+4ab}\right].
\end{split}
\eeqst
Therefore, \eqref{ip2thm1} holds iff 
$$
\lambda_n\ne \lambda_k +{\sqrt{(c-d)^2+4ab}\over\kappa},\;\; k,n\in\N,\, n>k.
$$
If $b=0$ then $\breve\lambda_k,\hat\lambda_k\in \{\kappa\lambda_k+c;\varkappa\lambda_k+d\}$, $k\in\N$. Thus, 
\eqref{ip2thm1} is valid iff
$$
\kappa\lambda_k+c\ne \varkappa\lambda_n+d,\;\; k,n\in\N,\, n\ne k.
$$

  %%%%%%%%%%%%%%%%%%%%%%%%%%%%%%%%

\section*{Appendix: Proof of Theorem \ref{dpthm1}}

Before proving Theorem \ref{dpthm1} we  provide a lemma concerning
 Prabhakar function 
$
E_{\alpha,\beta}^\gamma(z)={1\over\Gamma(\gamma)}\sum\limits_{n=0}^\infty {\Gamma(\gamma+n)z^n\over n!\Gamma(\alpha n+\beta)},
$ $\beta,\gamma\in\C$. 

\begin{lemma}\label{mlflem}
\begin{description}
\item{\rm (i)} If $\gamma>0$ then 
\beq\label{mlf1}
\forall \theta\in \big(0,\frac{(2-\alpha)\pi}{2}\big)\;\;\; \exists c_{\theta}\; :\; |E_{\alpha,\beta}^\gamma(-z)|\le {c_\theta\over (1+|z|)^\gamma},\; z\in\Sigma(0,\theta), 
\eeq
where $c_\theta$ is a $\theta$-dependent constant that also depends on $\alpha,\beta,\gamma$. If $\beta,\gamma,\lambda>0$ then
\beq\label{mlf2}
\left({\mathcal L}\left(t^{\beta-1}E_{\alpha,\beta}^\gamma(-\lambda t^\alpha)\right)\right)(s)={s^{\alpha\gamma-\beta}\over (s^\alpha+\lambda)^\gamma},\; {\rm Re}s>0.
\eeq
\item{\rm (ii)} Let  $w\in L_1(0,t_0)$, $\lambda>0$, $q\in\R$. Then the function
$$\phi(z)=\int_0^{t_0} (z-\tau)^q E_{\alpha,\beta}^\gamma(-\lambda (z-\tau)^\alpha)w(\tau)d\tau$$ 
is analytic in $\Sigma(t_0,\pi)$.
\end{description}
\end{lemma} 

\noindent {\it Proof}. (i) For \eqref{mlf2} see \cite{Garra}. The estimate \eqref{mlf1} follows from the asymptotic relation 
\beqst\label{mlfasym}
E_{\alpha,\beta}^\gamma(-z)\sim {1\over \Gamma(\beta-\alpha\gamma)}z^{-\gamma}+O\left({1\over |z|^{\gamma+1}}\right)\;\;\mbox{as}\; |z|\to \infty,\; z\in \Sigma(0,\theta),
\eeqst
that holds if $\theta\in (0,{(2-\alpha)\pi\over 2})$ (see \cite{Garra})
and the fact that $E_{\alpha,\beta}^\gamma(z)$ is locally bounded as an entire function. 

(ii) 
Since $z^q E_{\alpha,\beta}^\gamma(-\lambda z^\alpha)$ is analytic in $\Sigma(0,\pi)$, the assertion follows from Lemma 5 of \cite{Janno2024}.
\hfill $\Box$

\medskip
\noindent{\it Proof of Theorem} \ref{dpthm1}. (i)  Since 
$f=g=0$, $\varphi=\psi=0$, we have $F_{kl}={\mathcal X}_{kl}=0$, $\varphi_{kl}=\psi_{kl}=0$, $l=1,\ldots,\mu_k$, $k\in\N$. 
From formulas \eqref{prop1} and \eqref{prop2} we obtain  $U_{kl}=V_{kl}=0$, $l=1,\ldots,\mu_k$, $k\in\N$, hence $u_{kl}=v_{kl}=0$, $l=1,\ldots,\mu_k$, $k\in\N$. 
This implies $u=v=0$. 

\smallskip (ii) Let $k\in\N$, $l=1,\ldots,\mu_k$ and define the following  functions:
\beq\label{ukvk}
u_{kl}(t)=u_{kl}^1(t)+u_{kl}^2(t),\;\;v_{kl}(t)=v_{kl}^1(t)+v_{kl}^2(t)
\eeq
for $t>0$, where 
\beq\label{ukvk1uus}
&&\hskip-5truemm \begin{split}
&u_{kl}^1(t)=\left[E_{\alpha,1}^1(-\breve\lambda_kt^\alpha)+\vartheta_k q_k(t)\right]\varphi_{kl}-a q_k(t)\psi_{kl},
\\
&v_{kl}^1(t)=-bq_k(t)\varphi_{kl}+\left[E_{\alpha,1}^1(-\breve\lambda_kt^\alpha)+\zeta_k q_k(t)\right]\psi_{kl},
\end{split}
\\[1ex]\label{ukvk2uus}
&&\hskip-5truemm \begin{split}
&u_{kl}^2(t)=\left[t^{\alpha-1}E_{\alpha,\alpha}^1(-\breve\lambda_kt^\alpha)+\vartheta_k w_k\right]*f_{kl}(t)-a w_k*\chi_{kl}(t),
\\
&v_{kl}^2(t)=-bw_k*f_{kl}(t)+\left[t^{\alpha-1}E_{\alpha,\alpha}^1(-\breve\lambda_kt^\alpha)+\zeta_k w_k\right]*\chi_{kl}(t),
\end{split}
\eeq
\beq\label{thetak}
\vartheta_k={\varkappa-\kappa\over 2}\lambda_k+{d-c\over 2}+{\breve\lambda_k-\hat\lambda_k\over 2},\;\; \zeta_k={\kappa-\varkappa\over 2}\lambda_k+{c-d\over 2}+{\breve\lambda_k-\hat\lambda_k\over 2},
\eeq
$*$ denotes the convolution, i.e.
 $\phi_1*\phi_2(t)=\int_0^t \phi_1(t-\tau)\phi_2(\tau)d\tau$, and the functions $q_k$, $w_k$ extended to a complex variable $z$ are defined as follows:
\beq\label{qk}
&&q_k(z)=\left\{\begin{array}{ll} 
{E_{\alpha,1}^1(-\hat\lambda_k z^\alpha)-E_{\alpha,1}^1(-\breve\lambda_k z^\alpha)\over \breve\lambda_k-\hat\lambda_k}\; &\mbox{if $\breve\lambda_k\ne \hat\lambda_k$},
\\[1ex]
z^\alpha E_{\alpha,\alpha+1}^2 (-\breve\lambda_k z^\alpha)\; &\mbox{if $\breve\lambda_k= \hat\lambda_k$},
\end{array}\right.
\eeq
\beq\label{wk}
&&w_k(z)=\left\{\begin{array}{ll} 
{z^{\alpha-1}E_{\alpha,\alpha}^1(-\hat\lambda_k z^\alpha)-z^{\alpha-1}
E_{\alpha,\alpha}^1(-\breve\lambda_k z^\alpha)\over \breve\lambda_k-\hat\lambda_k}\; &\mbox{if $\breve\lambda_k\ne \hat\lambda_k$},
\\[1ex]
z^{2\alpha-1} E_{\alpha,2\alpha}^2 (-\breve\lambda_k z^\alpha)\; &\mbox{if $\breve\lambda_k= \hat\lambda_k$}.
\end{array}\right.
\eeq
Since ${d\over dz}E^1_{\alpha,1}(z)=E_{\alpha,\alpha+1}^2(z)$, ${d\over dz}E^1_{\alpha,\alpha}(z)=E_{\alpha,2\alpha}^2(z)$ \cite{Garra}, we can express $q_k$ and
$w_k$ in another form:
\beq\label{qkwkshort}
q_k(z)= \int_{\hat\lambda_k}^{\breve\lambda_k}{z^\alpha E_{\alpha,\alpha+1}^2 (-\xi z^\alpha)\over \breve\lambda_k-\hat\lambda_k}d\xi,\;\;
w_k(z)=\int_{\hat\lambda_k}^{\breve\lambda_k}{z^{2\alpha-1} E_{\alpha,2\alpha}^2 (-\xi z^\alpha)\over \breve\lambda_k-\hat\lambda_k}d\xi.
\eeq

By means of \eqref{lamineq}, \eqref{mlf1}, \eqref{thetak}, \eqref{qkwkshort} and the relation $1\le {\lambda_k\over\lambda_1}$ we deduce the following estimates:
\beq\label{thm1r1an}
&&\hskip-1.2truecm \begin{split}
&|E_{\alpha,1}^1(-\breve\lambda_k z^\alpha)|\le {C_\theta\over 1+\lambda_k|z|^\alpha},\; z\in\Sigma(0,\theta),
\\
&|q_k(z)|,|\vartheta_k q_k(z)|,|\zeta_k q_k(z)|\le {C_\theta\lambda_k|z|^\alpha\over (1+\lambda_k|z|^\alpha)^2}
\le {C_\theta\over 1+\lambda_k|z|^\alpha},\;  z\in\Sigma(0,\theta),
\end{split}
\\[1ex]
\label{thm1r3an}
&&\hskip-1.2truecm \begin{split}
&|z^{\alpha-1}E_{\alpha,\alpha}^1(-\breve\lambda_k z^\alpha)|\le {C_\theta |z|^{\alpha-1}\over 1+\lambda_k|z|^\alpha},\;z\in\Sigma(0,\theta),
\\
&|w_k(z)|,|\vartheta_k w_k(z)|,|\zeta_k w_k(z)|\le {C_\theta\lambda_k |z|^{2\alpha-1}\over (1+\lambda_k|z|^\alpha)^2}
\le {C_\theta |z|^{\alpha-1}\over 1+\lambda_k|z|^\alpha},\; z\in\Sigma(0,\theta),
\end{split}
\eeq
for any $\theta\in (0,{(2-\alpha)\pi\over 2\alpha})$, where $C_\theta$ is a constant.

Let us estimate $u_{kl}$ and $v_{kl}$. By means of  \eqref{thm1r1an} from \eqref{ukvk1uus}
 we deduce
\beq\label{thm1vvv}
|u_{kl}^1(t)|+|v_{kl}^1(t)|\le {C_2\over 1+\lambda_kt^\alpha}(|\varphi_{kl}|+|\psi_{kl}|),\; t>0,
\eeq
with a constant $C_2$. Without loss of generality we may assume that $r\le 1$ in \eqref{phiass}. We obtain
\beq\nonumber
&&|u_{kl}^1(t)|+|v_{kl}^1(t)|\le C_2 t^{\alpha(r-1)}\lambda_k^{r-1}{(\lambda_kt^\alpha)^{1-r}\over 1+\lambda_kt^\alpha}(|\varphi_{kl}|+|\psi_{kl}|)
\\ \label{thm1r2}
&&\le C_2 t^{\alpha(r-1)}\lambda_k^{r-1}(|\varphi_{kl}|+|\psi_{kl}|),\; t>0.
\eeq
By means of  \eqref{thm1r3an} from \eqref{ukvk2uus} we have
\beq\label{thm1www}
|u_{kl}^2(t)|+|v_{kl}^2(t)|\le C_3\left({t^{\alpha-1}\over 1+\lambda_kt^\alpha}\right)*(|f_{kl}(t)|+|\chi_{kl}(t)|)\; t>0,
\eeq
with a constant $C_3$. Without loss of generality assume that $r_1\le 1$ in \eqref{fass2}. We obtain
\beq\nonumber
&&|u_{kl}^2(t)|+|v_{kl}^2(t)|\le C_3 \left(t^{\alpha r_1-1}\lambda_k^{r_1-1}{(\lambda_kt^\alpha)^{1-r_1}\over 1+\lambda_kt^\alpha}\right)*(|f_{kl}(t)|+|\chi_{kl}(t)|)
\\ \label{thm1r4}
&&\le C_3 \lambda_k^{r_1-1}t^{\alpha r_1-1}*(|f_{kl}(t)|+|\chi_{kl}(t)|),\; t>0.
\eeq

The estimates \eqref{thm1vvv} and \eqref{thm1www} imply the assertion \eqref{thm1k}. 
Next let us define $u=u^1+u^2$ and $v=v^1+v^2$, where
\beq\label{udef}
  u^j=\sum_{k=1}^\infty\sum_{l=1}^{\mu_k} u_{kl}^j{\rm z}_{kl},\; v^j=\sum_{k=1}^\infty\sum_{l=1}^{\mu_k}v_{kl}^j{\rm z}_{kl},\; j=1,2.
\eeq
By means of \eqref{thm1r2}   we obtain 
\beq\nonumber
&&\|u^1\|_{L_p((0,T);{\mathcal D}^{1})}=\Big\| \Big[\sum_{k=1}^\infty\sum_{l=1}^{\mu_k} \lambda_k^{2}|u_{kl}^1(t)|^2\Big]^{1\over 2} \Big\|_{L_p(0,T)}
\\ \nonumber
&&\le C_{2}\Big\|\Big[\sum_{k=1}^\infty\sum_{l=1}^{\mu_k} t^{2\alpha(r-1)}\lambda_k^{2r}(|\varphi_{kl}|+|\psi_{kl}|)^2\Big]^{1\over 2} \Big\|_{L_p(0,T)}
\\ \label{thm1q1}
&&\le C_{2} \|t^{\alpha(r-1)}\|_{L_p(0,T)}(\|\varphi\|_{{\mathcal D}^{r}}+\|\psi\|_{{\mathcal D}^{r}})<\infty,\, T>0.
\eeq
 Similarly we deduce the estimates 
\beq\label{thm1q2a}
&&\hskip-1.3truecm \|v^1\|_{L_p((0,T);{\mathcal D}^{1})}\le C_{2} \|t^{\alpha(r-1)}\|_{L_p(0,T)}(\|\varphi\|_{{\mathcal D}^{r}}\!+\!\|\psi\|_{{\mathcal D}^{r}})\!<\!\infty,\, T>0,
\eeq
and 
\beq
\nonumber
&&\hskip-1.3truecm
\|e^{-\sigma t}u^1\|_{L_1((0,\infty);{\mathcal D}^{1})}, \|e^{-\sigma t}v^1\|_{L_1((0,\infty);{\mathcal D}^{1})}
\\  \label{thm1q2b}
&&\hskip-1.3truecm \le C_{2}\|e^{-\sigma t}t^{\alpha(r-1)}\|_{L_1(0,\infty)}(\|\varphi\|_{{\mathcal D}^{r}}+\|\psi\|_{{\mathcal D}^{r}})<\infty,\, \sigma>0.
\eeq
Further, by means of \eqref{thm1r4}   we have
\beq\nonumber
&&\hskip-1truecm \|u^2\|_{L_\infty((0,T);{\mathcal D}^{1})}=\Big\|\Big[\sum_{k=1}^\infty \sum_{l=1}^{\mu_k}
\lambda_k^{2}|u_{kl}^2(t)|^2\Big]^{1\over 2} \Big\|_{L_\infty(0,T)}
\\ \nonumber
&&\hskip-1truecm\le C_3 \Big\|\Big[\sum_{k=1}^\infty\sum_{l=1}^{\mu_k} \lambda_k^{2r_1}\big[\big(t^{\alpha r_1-1}*|f_{kl}|(t)\big)^2+\big(t^{\alpha r_1-1}*|\chi_{kl}|(t)\big)^2\big]\Big]^{1\over 2} \Big\|_{L_\infty(0,T)}
\\ \label{thm1q3}
&&\hskip-1truecm\le C_3{t_0^{\alpha r_1}\over \alpha r_1} \Big[\sum_{k=1}^\infty \sum_{l=1}^{\mu_k}\lambda_k^{2r_1}\left(\|f_{kl}\|_{L_\infty(0,t_0)}^2+
\|\chi_{kl}\|_{L_\infty(0,t_0)}^2\right)\Big]^{1\over 2}<\infty,\, 
\\ \nonumber
&&\hskip-1truecm T>0.
\eeq
Similarly we show that 
\beq \nonumber
&&\hskip-1truecm\|v^2\|_{L_\infty((0,T);{\mathcal D}^{1})}\le  C_3{t_0^{\alpha r_1}\over \alpha r_1}
\\  \label{thm1q4}
&&\hskip-1truecm  \times\Big[\sum_{k=1}^\infty \sum_{l=1}^{\mu_k}\lambda_k^{2r_1}\left(\|f_{kl}\|_{L_\infty(0,t_0)}^2+
\|\chi_{kl}\|_{L_\infty(0,t_0)}^2\right)\Big]^{1\over 2}<\infty, T>0,
\eeq
and
\beq
\nonumber
&&\hskip-1.3truecm
\|e^{-\sigma t}u^2\|_{L_1((0,\infty);{\mathcal D}^{1})}, \|e^{-\sigma t}v^2\|_{L_1((0,\infty);{\mathcal D}^{1})}\le C_3{1\over\sigma} {t_0^{\alpha r_1}\over \alpha r_1}
\\  \label{thm1q4b}
&&\hskip-1.3truecm  \times\Big[\sum_{k=1}^\infty \sum_{l=1}^{\mu_k}\lambda_k^{2r_1}\left(\|f_{kl}\|_{L_\infty(0,t_0)}^2+
\|\chi_{kl}\|_{L_\infty(0,t_0)}^2\right)\Big]^{1\over 2}<\infty, \sigma>0.
\eeq

The relations \eqref{thm1q1} - \eqref{thm1q4b} imply 
$(u,v)\in (L_p((0,T);H_0^2(\Omega))^2$, $T>0$, and the assertion \eqref{uexp}.

Next we are going show that $(u,v)$ belongs to $({\mathcal W}_{\alpha,p})^2$ and solves \eqref{e1} - \eqref{e6}. 
Let us apply Laplace transform to equalities \eqref{ukvk} - \eqref{ukvk2uus}. Using \eqref{mlf2}, \eqref{qk}, \eqref{wk} and the convolution rule, we obtain
\beqst
\begin{array}{ll}
&U_{kl}(s)=\Big({1\over s^\alpha+\breve\lambda_k}+{\vartheta_k\over (s^\alpha+\hat\lambda_k)(s^\alpha+\breve\lambda_k)}\Big)(F_{kl}(s)+s^{\alpha-1}\varphi_{kl})
\\
&\qquad -
{a\over (s^\alpha+\hat\lambda_k)(s^\alpha+\breve\lambda_k)}({\mathcal X}_{kl}(s)+s^{\alpha-1}\psi_{kl}),
\\[2ex]
&V_{kl}(s)=-
{b\over (s^\alpha+\hat\lambda_k)(s^\alpha+\breve\lambda_k)}(F_{kl}(s)+s^{\alpha-1}\varphi_{kl})
\\
&\qquad +\Big({1\over s^\alpha+\breve\lambda_k}+{\zeta_k\over (s^\alpha+\hat\lambda_k)(s^\alpha+\breve\lambda_k)}\Big)({\mathcal X}_{kl}(s)+s^{\alpha-1}\psi_{kl})
\end{array}
\eeqst
for ${\rm Re}s>0$. Using in these equalities the relations \eqref{factor}, \eqref{lambdad} and \eqref{thetak} we obtain the system \eqref{prop1}, \eqref{prop2}
 for ${\rm Re}s>0$. Solving the latter one for
$F_{kl}(s)+s^{\alpha-1}\varphi_{kl}$ and ${\mathcal X}_{kl}(s)+s^{\alpha-1}\psi_{kl}$
and rearranging the terms in obtained expressions in a suitable way, we reach \eqref{pro5}, \eqref{pro6}. Applying  the inverse Laplace transform to \eqref{pro5}, \eqref{pro6} we reach the system 
\eqref{pro3}, \eqref{pro4} for $(u_{kl},v_{kl})$. This immediately implies the system \eqref{pro1}, \eqref{pro2} for $(u,v)$. Let $T>0$. Due to the proved relation 
$(u,v)\in (L_p((0,T);H_0^2(\Omega))^2$ and imposed assumptions on $f$, $\chi$, the arguments of the operator $I^\alpha$ on the right-hand sides of 
\eqref{pro1}, \eqref{pro2} belongs to $L_p((0,T);L_2(\Omega))$. Due to Lemma \ref{lemma31}, it holds $u-\varphi,v-\psi\in {_0H^\alpha_p}((0,T);L_2(\Omega))$ and 
the equations \eqref{pro1}, \eqref{pro2} imply \eqref{e1}, \eqref{e2}. Moreover, since $p\in ({1\over\alpha},\infty)$, Lemma \ref{lemma31} implies $u,v\in C([0,T];L_2(\Omega))$
and $(u(t,\cdot)-\varphi)|_{t=0}=(v(t,\cdot)-\psi)|_{t=0}=0$. The latter one implies \eqref{e5}, \eqref{e6}. Summing up, the proved properties of $u$ and $v$ yield 
$(u,v)\in ({\mathcal W}_{\alpha,p,T})^2$. Since $T>0$ is arbitrary, we have $(u,v)\in ({\mathcal W}_{\alpha,p})^2$.

To complete the proof of (ii), it remains to verify the analytic continuation assertion. 
Let us define $u_{kl}(z)=u_{kl}^1(z)+u_{kl}^2(z)$, $v_{kl}(z)=v_{kl}^1(z)+v_{kl}^2(z)$ for $z\in \Sigma (t_0,\pi)$, where
\beqst
&&\hskip-5truemm \begin{split}
&u_{kl}^1(z)=\left[E_{\alpha,1}^1(-\breve\lambda_kz^\alpha)+\vartheta_k q_k(z)\right]\varphi_{kl}-a q_k(z)\psi_{kl},
\\
&v_{kl}^1(z)=-bq_k(z)\varphi_{kl}+\left[E_{\alpha,1}^1(-\breve\lambda_kz^\alpha)+\zeta_k q_k(z)\right]\psi_{kl},
\end{split}
\\[1ex]
&&\hskip-5truemm \begin{split}
&u_{kl}^2(z)=\left[z^{\alpha-1}E_{\alpha,\alpha}^1(-\breve\lambda_kz^\alpha)+\vartheta_k w_k\right]\star f_{kl}(z)-a w_k\star\chi_{kl}(z),
\\
&v_{kl}^2(z)=-bw_k\star f_{kl}(z)+\left[z^{\alpha-1}E_{\alpha,\alpha}^1(-\breve\lambda_kz^\alpha)+\zeta_k w_k\right]\star \chi_{kl}(z)
\end{split}
\eeqst
and $\phi_1\star\phi_2(z)=\int_0^{t_0}\phi_1(z-\tau)\phi_2(\tau)d\tau$. Since $f_{kl}(\tau)=\chi_{kl}(\tau)=0$, $\tau>t_0$, the functions $u_{kl}(z)$, $v_{kl}(z)$ are the extensions of 
$u_{kl}(t)$, $v_{kl}(t)$, $t>t_0$. Observing that 
Prabhakar functions are entire, Lemma \ref{mlflem} and \eqref{qk}, \eqref{wk} we see that $u_{kl}^j(z)$, $v_{kl}^j(z)$, $j=1,2$, are analytic in 
$\Sigma (t_0,\pi)$.

Next let us define the functions
$$
{\rm u}_{K}(z)\!=\!\sum_{k=1}^K\sum_{l=1}^{\mu_k}(u_{kl}^1(z)+u_{kl}^2(z)){\mathrm z}_{kl},\, {\rm v}_{K}(z)\!=\!\sum_{k=1}^K\sum_{l=1}^{\mu_k}(v_{kl}^1(z)+v_{kl}^2(z)){\mathrm z}_{kl},\, K\in\N.
$$
These are analytic in $\Sigma (t_0,\pi)$ as functions with values in $H_0^2(\Omega)$.

 Let $\theta\in (0,{(2-\alpha)\pi\over 2\alpha})$ and $D$ be a compact subset of $\Sigma (t_0,\theta)$.
Then  $\inf\limits_{z\in D}|z|>0$. 
By means of \eqref{thm1r1an} we obtain
\beq
\label{thm1prlopp1}
\lambda_k (|u_{kl}^1(z)|+|v_{kl}^1(z)|)\le {C_D^1 \over |z|^\alpha}{\lambda_k |z|^\alpha\over 1+\lambda_k|z|^\alpha}(|\varphi_{kl}|+|\psi_{kl}|)
\le C_D^2 (|\varphi_{kl}|+|\psi_{kl}|)
\eeq
for $z\in D$ with some constants $C_D^1$ and $C_D^2$. Moreover, $z-\tau\in \Sigma (0,\theta)$ for $z\in D$, $0<\tau<t_0$ 
and $\inf\limits_{z\in D,0<\tau<t_0}|z-\tau|>0$. By means of \eqref{thm1r3an} we deduce
\beq\nonumber
&&\hskip-5truemm \lambda_k (|u_{kl}^2(z)|+|v_{kl}^2(z)|)\le C_D^3\int_0^{t_0} |z-\tau|^{-1}{\lambda_k|z-\tau|^\alpha\over 1 +\lambda_k|z-\tau|^\alpha}d\tau
\\ \label{thm1prlopp2}
&&\hskip-5truemm\times (\|f_{kl}\|_{L_\infty(0,t_0)}+\|\chi_{kl}\|_{L_\infty(0,t_0)})\le C_D^4 (\|f_{kl}\|_{L_\infty(0,t_0)}+\|\chi_{kl}\|_{L_\infty(0,t_0)})
\eeq
for $z\in D$ with constants $C_D^3$ and $C_D^4$. The estimates \eqref{thm1prlopp1}, \eqref{thm1prlopp2} imply that
\beqst
&&\|{\rm u}_{K_2}(z)-{\rm u}_{K_1}(z)\|_{H_0^2(\Omega)}, \|{\rm v}_{K_2}(z)-{\rm v}_{K_1}(z)\|_{H_0^2(\Omega)}
\\
&&\le C_D\sum_{k=K_1}^{K_2}\sum_{l=1}^{\mu_k}(|\varphi_{kl}|+|\psi_{kl}|+
\|f_{kl}\|_{L_\infty(0,t_0)}+\|\chi_{kl}\|_{L_\infty(0,t_0)})
\eeqst
for $K_1,K_2\in\N$, $K_2>K_1$, where $C_D$ is a constant. Due to the assumptions \eqref{phiass} and \eqref{fass2}, ${\rm u}_{K}(z)$ and ${\rm v}_{K}(z)$ are Cauchy
sequences  for $z\in D$. Therefore,  the series $u(z,\cdot)=\sum\limits_{k=1}^\infty\sum\limits_{l=1}^{\mu_k}(u_{kl}^1(z)+u_{kl}^2(z)){\mathrm z}_{kl}$ 
and
$v(z,\cdot)=\sum\limits_{k=1}^\infty\sum\limits_{l=1}^{\mu_k}(v_{kl}^1(z)+v_{kl}^2(z)){\mathrm z}_{kl}$ exist for $z\in D$. Moreover, these series are uniformly convergent in $D$. Since $D$  is an arbitrary compact subset of $\Sigma (t_0,\theta)$, $\theta\in (0,{(2-\alpha)\pi\over 2\alpha})$, the functions $u(z,\cdot)$ and $v(z,\cdot)$ are
analytic in $\Sigma(t_0,{(2-\alpha)\pi\over 2\alpha})$  with values in $H_0^2(\Omega)$.
  The restriction to $(t_0,\infty)$ of $(u(z,\cdot),v(z,\cdot))$  coincides with the 
 solution of \eqref{e1} - \eqref{e6}. 
\hfill $\Box$

\end{document}